\documentclass[12pt,a4paper]{article}

\usepackage{hyperref}
\usepackage{amsfonts}
\usepackage{mathrsfs}
\usepackage{amsmath}
\usepackage{amssymb}
\usepackage{amsthm}
\usepackage{amscd}
\usepackage{graphicx}
\usepackage{indentfirst}
\usepackage[all]{xy}
\usepackage[top=21mm, bottom=21mm, left=20mm, right=20mm]{geometry}
\usepackage{enumerate}
\usepackage{bm}
\usepackage{enumitem}
\usepackage{xcolor}

\newtheorem{theorem}{Theorem}[section]
\newtheorem{lemma}[theorem]{Lemma}
\newtheorem{proposition}[theorem]{Proposition}
\newtheorem{corollary}[theorem]{Corollary}

\newtheorem{claim}[theorem]{Claim}
\newtheorem{example}[theorem]{Example}
\newtheorem{question}[theorem]{Question}
\newtheorem{problem}[theorem]{Problem}
\DeclareMathOperator{\rank}{rank}
\newcommand{\ma}{\mathcal}

\newcommand{\s}{\subseteq}

\newcommand{\fr}{\frac}

\begin{document}

\title{The hat guessing number of graphs\footnote{Part of this paper was presented in 2019 IEEE International Symposium on Information Theory \cite{alon2019hat}.}}
\author{Noga Alon\footnote{Department of Mathematics,
Princeton University, Princeton, NJ 08544, USA
and
Schools of Mathematics and Computer Science, Tel Aviv University,
Tel Aviv 6997801, Israel.
Email: nogaa@tau.ac.il.},
Omri Ben-Eliezer\footnote{Blavatnik School of Computer Science, Tel Aviv University, Tel Aviv 6997801, Israel. Email: omrib@mail.tau.ac.il.},
Chong Shangguan\footnote{Department of Electrical Engineering-Systems, Tel Aviv University, Tel Aviv 6997801, Israel. Email: theoreming@163.com.},
and Itzhak Tamo\footnote{Department of Electrical Engineering-Systems, Tel Aviv University, Tel Aviv 6997801, Israel. Email: zactamo@gmail.com.}
}
\maketitle

\begin{abstract}
Consider the following hat guessing game:
$n$ players are placed on $n$ vertices of a graph, each wearing a hat whose color is arbitrarily chosen from a set of $q$ possible colors.
Each player can see the hat colors of his neighbors, but not his own hat color.
All of the players are asked to guess their own hat colors simultaneously, according to a predetermined guessing strategy and the hat colors they see, where no communication between them is allowed.
Given a graph $G$, its hat guessing number ${\rm{HG}}(G)$ is the largest integer $q$ such that there exists a guessing strategy guaranteeing at least one correct guess for any hat assignment of $q$ possible colors.

In 2008, Butler et al. asked whether the hat guessing number of
the complete bipartite graph $K_{n,n}$ is at least some fixed positive
(fractional) power of $n$.
We answer this question affirmatively, showing that for sufficiently large $n$, the complete $r$-partite graph $K_{n,\ldots,n}$ satisfies ${\rm{HG}}(K_{n,\ldots,n})=\Omega(n^{\frac{r-1}{r}-o(1)})$.
Our guessing strategy is based on a probabilistic construction and
other combinatorial ideas, and can be extended to show
that ${\rm{HG}}(\vec{C}_{n,\ldots,n})=\Omega(n^{\frac{1}{r}-o(1)})$,
where $\vec{C}_{n,\ldots,n}$ is the blow-up of a directed
$r$-cycle,
and where for directed graphs each player sees only the hat colors
of his outneighbors.

Additionally, we consider related problems like the relation between the hat guessing number and other graph parameters, and the linear hat guessing number, where the players are only allowed to use affine linear guessing strategies.
Several nonexistence results are obtained by using well-known combinatorial tools, including the Lov\'asz Local Lemma and the Combinatorial Nullstellensatz.
Among other results, it is shown that under certain conditions, the
linear hat guessing number of $K_{n, n}$ is at most $3$,
exhibiting a huge gap from the $\Omega(n^{\fr{1}{2} - o(1)})$
(nonlinear)
hat guessing number of this graph.
\end{abstract}

\section{Introduction}

\noindent Hat guessing problems are interesting recreational mathematical puzzles that have attracted a lot of attention throughout the years.
A classical variant \cite{feige2004you, winkler2002games} involves $n \geq 2$ players, each wearing a hat in $q \geq 2$ possible colors assigned to it by an adversary.
Each player sees the hat colors of all the other players, but not his own, and based on this information he makes a guess on his own hat color.
The goal of the players is to ensure that at least one player will make a correct guess, regardless of the hat assignment by the adversary.
The players are allowed to communicate and pick a guessing strategy only \emph{before} the hats are assigned, and no communication is allowed afterwards.
Once \emph{all} the players made their guesses, the adversary verifies whether there was a player who guessed correctly, and if so,  we say that the players {\it win}.

With the above rules in mind, the puzzle asks what is the maximum number of hat colors $q$ for which the players have a winning guessing strategy.
Perhaps surprisingly, the answer to this question is $q=n$.
Indeed, number the players and the hat colors with the numbers $0,1\ldots,n-1$, and let player $i$ guess that his color is the unique color for which the sum of all the hat colors (including his hat color) modulo $n$ is $i$.
It is not hard to verify that exactly one of the players guesses correctly, regardless of the coloring assigned by the adversary.
A more general statement on this problem for arbitrary $n,q$,
observed by Feige \cite{feige2004you},
claims that the players can always ensure that
at least $\lfloor n / q \rfloor$ of them
guess correctly, and that this is tight.

A natural generalization of the above puzzle asks the same question but assumes that each player can only see some subset of the other players' hat colors.
This generalization, which is the problem considered in this paper, was first presented by Butler et al.\@ \cite{Butler08} and further investigated in a line of other works \cite{gadouleau2018finite, Gadouleau09, Szczechla17}.
A formal definition of the problem is as follows.
Let $G$ be a simple graph on $n$ vertices $\{v_1,\ldots,v_n\}$, and let $Q$ be a finite set of $q$ colors.
The $n$ vertices of the graph are identified with the $n$ players, where each is assigned arbitrarily with a hat colored with one of the colors in $Q$.
A player can only see the hat colors of his neighbors, i.e., player $i$ sees the hat color of player $j$ if and only if $v_i$ is connected to $v_j$ in $G$.
After all of the players agreed on a guessing strategy, they are asked to guess their own hat colors simultaneously, and no communication of any sort is allowed at this point.
The goal of the players is to ensure that at least one player guesses his hat color correctly.

The hat guessing problem is completely defined by the graph $G$ which is called the {\it sight graph}.
Therefore, for a given graph $G$, its {\it hat guessing number} ${\rm{HG}}(G)$, as defined by Farnik \cite{Farnik2015}, is the largest positive integer $q$ such that there exists a winning guessing strategy for the players.
If ${\rm{HG}}(G)\ge q$, $G$ is also called {\it $q$-solvable} by Gadouleau and Georgiou \cite{Gadouleau09}.

In general, the sight graph $G$ may be directed; a directed edge $v_i\rightarrow v_j$ represents that player $i$ can see the hat color of player $j$.
In the sequel we do not distinguish between the vertices and the players.
The color of a vertex and its guessing strategy refer to the hat color of the corresponding player and his guessing strategy.

This paper focuses on the graph parameter ${\rm{HG}}(G)$, and it provides improved upper and lower bounds on ${\rm{HG}}(G)$ for several graph families.
In the literature there are only a few graphs whose hat
guessing numbers have been determined precisely.
Below we list all of them.
As mentioned earlier, for the complete graph $K_n$ we have ${\rm{HG}}(K_n)=n$ \cite{feige2004you}.
Butler et al. \cite{Butler08} showed that all trees are not 3-solvable, implying ${\rm{HG}}(T)=2$ for any tree $T$.
Lastly, Szczechla \cite{Szczechla17} recently showed that a cycle of length $n$ is 3-solvable if and only if $n=4$ or is a multiple of $3$, and that all cycles are not $4$-solvable.

Next we state our main results, while we delay some of the needed definitions to Section \ref{sec:prelims}.
The various variants and models of hat guessing problems are reviewed in Section \ref{sec:related}. Note that all asymptotics are in $n$ and we omit all floor and ceiling signs whenever these are not crucial.

\subsection{Complete multipartite graphs and digraphs}

\noindent Let $H$ be a subgraph of a graph $G$; it is clear that ${\rm{HG}}(H)\le {\rm{HG}}(G)$.
Furthermore, since complete graphs have large hat guessing numbers \cite{feige2004you} it follows that graphs which contain large cliques as subgraphs also have large hat guessing numbers.
It is thus an interesting question to ask whether the {\it clique number} of $G$, which is the number of vertices in a maximum complete subgraph in $G$, determines its hat guessing number.
In other words, how large can ${\rm{HG}}(G)$ be if its clique number is bounded from above by a constant.
In \cite{Butler08} it was shown that the complete bipartite graph $K_{q-1,q^{q^{q-1}}}$ is $q$-solvable, implying that for large $n$, ${\rm{HG}}(K_{n,n})=\Omega(\log\log n)$, while the clique number is clearly $2$.
The value of ${\rm{HG}}(K_{n,n})$ was further considered in \cite{Gadouleau09}, where it was shown that ${\rm{HG}}(K_{m,n})\le\min\{m+1,n+1\}$ and ${\rm{HG}}(K_{q-1,(q-1)^{q-1}})\ge q$, implying that $\Omega(\log n)={\rm{HG}}(K_{n,n})\le n+1$.
Following these results it is natural to consider the question below, which is originally posed in \cite{Butler08}.

\begin{question}\label{question1}
 Does there exist a constant $\alpha>0$ independent of $n$, such that ${\rm{HG}}(K_{n,n})\ge n^{\alpha}$ for sufficiently large $n$?
\end{question}

We answer this question affirmatively in the following generalized sense.

\begin{theorem}\label{polynomialbipartite}
  For integers $r\ge 2,q\ge 2$, let $K_{m,\ldots,m,n}$ be the complete $r$-partite graph in which there are $r-1$ vertex parts of size $m$ and one vertex part of size $n$.
  Then there exists a constant $c$ not depending on $q$ such that for $m=(2q\ln q)^{\fr{1}{r-1}}$ and $n=cr(q\ln q)^{\fr{r}{r-1}}$,
  $${\rm{HG}}(K_{m,\ldots,m,n})\ge q,$$
  which implies that 
  ${\rm{HG}}(K_{n,\ldots,n})\ge n^{\fr{r-1}{r}-o(1)}$, where $K_{n,\ldots,n}$ is the complete $r$-partite graph of equal part size $n$.
\end{theorem}

In particular, by combining \cite{Gadouleau09} and Theorem \ref{polynomialbipartite} we have that $\Omega(n^{\fr{1}{2}-o(1)})={\rm{HG}}(K_{n,n})\le n+1$.
The determination of the exact value of ${\rm{HG}}(K_{n,n})$ is left as an interesting open question.

In \cite{Gadouleau09} the hat guessing number of directed graphs, which is somewhat less understood than that of undirected graphs, was considered.
Specifically, \cite{Gadouleau09} asked whether there exists an oriented graph with hat guessing number greater than $4$, where an {\it oriented graph} is a directed graph such that none of its pairs of vertices $\{u,v\}$ is connected by two symmetric directed edges $u\rightarrow v$ and $v\rightarrow u$.
Recently, Gadouleau \cite{gadouleau2018finite} provided a positive answer to this question, where he showed that for any $g\ge 3$ and sufficiently large $q$, there exists a $q$-solvable oriented graph with girth  $g$ and $q^{(1+o(1)) (g-1) \ln g}$ vertices.
In Theorem \ref{polynomialdirected} below, we provide another construction of a $q$-solvable oriented graph with girth $g$ and    $q^{(1+o(1))g}$ vertices, which for $g\ge 4$ is a slight improvement over \cite{gadouleau2018finite} on the number of vertices needed in a graph with these properties.

For $r\ge 3$, let $\vec{C}_r$ be the directed cycle on $r$ vertices $v_1\rightarrow v_2\cdots\rightarrow v_r\rightarrow v_1$.
The directed graph $\vec{C}_{n_1,\ldots,n_r}$ is obtained by replacing each vertex $v_i$ of $\vec{C}_r$ with a set $V_i$ of $n_i$ vertices, such that for any $u\in V_i,w\in V_j$, $u\rightarrow w$ if and only if $i\neq j$ and $v_i\rightarrow v_j$.
In other words, $\vec{C}_{n_1,\ldots,n_r}$ is obtained by blowing up each directed edge of $\vec{C}_r$ to a complete directed bipartite graph which preserves the direction of the original edge.
We call graphs of this type {\it complete $r$-partite directed cycles}.
With the above notation we have the following Theorem.

\begin{theorem}\label{polynomialdirected}
For integers $r\ge 3,q\ge 2$ and
$n_i = (r-1)\ln(2q\ln q)(4\ln q)^{r-i}q^{r+1-i}$ for $1\le i\le r$,
it holds that
$${\rm{HG}}(\vec{C}_{n_1,\ldots,n_r})\ge q,$$
which implies that 
${\rm{HG}}(\vec{C}_{n,\ldots,n})=\Omega(n^{\fr{1}{r}-o(1)})$, where $\vec{C}_{n,\ldots,n}$ is the complete $r$-partite directed cycle of equal part size $n$.
\end{theorem}



\subsection{Hat guessing number and other graph parameters}

\noindent In order to improve the understanding of ${\rm{HG}}(G)$, it is natural to try to relate it to other graph parameters of $G$, see e.g., \cite{Farnik2015,gadouleau2018finite}.
The maximum/minimum degree and the degeneracy are among the most basic parameters of a graph.
Therefore, we would like to understand how these parameters affect the hat guessing number, by answering the following questions.

\begin{problem}\label{degree}
Do there exist functions $f_i: \mathbb{N}\rightarrow\mathbb{N},~1\le i\le 3$ such that
\begin{itemize}
    \item [$(i)$] if the maximum degree of $G$ is $\Delta$, then ${\rm{HG}}(G)\le f_1(\Delta)$;
    \item [$(ii)$] if $G$ is $d$-degenerate, then ${\rm{HG}}(G)\le f_2(d)$;
    \item [$(iii)$] if the minimum degree of $G$ is $\delta$,
then ${\rm{HG}}(G)\ge f_3(\delta)$, and $f_3(\delta)$ tends to
infinity when $\delta$ tends to infinity.
\end{itemize}
\end{problem}

Currently, only $f_1$ is known to exist, i.e., the hat guessing number is bounded from above by a function of the maximum degree of the graph.
This result, as stated in the following theorem, is folklore \cite{Farnik2015}, and is a straightforward application of the Lov\'asz Local Lemma.

\begin{theorem}[Folklore]\label{maxdeg}
  Let $G$ be a graph with maximum degree $\Delta$, then ${\rm{HG}}(G)< e\Delta$.
\end{theorem}

By considering the hat guessing numbers of complete graphs and cycles one might conjecture that $f_1$ could be as small as $f_1(\Delta)=\Delta+1$.
Similarly, by considering the known upper bounds on the hat guessing
numbers of trees and complete bipartite graphs, one may
suspect that $f_2$ could also be as small as $f_2(d)=d+1$.
However, as opposed to $f_1$ it is not even clear  whether $f_2$ actually exists.
Theorem \ref{onemoreobservation} and Theorem \ref{trees} below, can be viewed as attempts  toward providing an answer to Problem \ref{degree} $(ii)$.
Although Theorem \ref{onemoreobservation} does not directly connect between the degeneracy and the hat guessing number, it shows that using the Lov\'asz Local Lemma one can obtain an upper bound on the hat guessing number, provided that the graph satisfies an additional property.
On the other hand, Theorem \ref{trees} does connect between these two graph parameters, since Corollary \ref{stamstam} which follows from it, shows that $f_2(1)\le 2$.

\begin{theorem}\label{onemoreobservation}
Let $k,d,q$ be integers and $G$ be a graph such that
\begin{itemize}
  \item [$(i)$] the induced subgraph of $G$ on all vertices of degree larger than $k$ is not $q$-solvable;
  \item [$(ii)$] each vertex in the induced subgraph of $G$ on all vertices of degree at most $k$ has at most $d$ vertices of distance at most $2$ from it.
\end{itemize} Then ${\rm{HG}}(G)<edq^k$.
\end{theorem}

Observe that $d \leq k^2$ always holds under the assumptions of Theorem \ref{onemoreobservation}. Hence, we get the following corollary as a special case.

\begin{corollary}\label{onemoreobservation2}
  Let $k, q$ be integers and $G$ be a graph whose induced subgraph on the vertices of degree larger than $k$ is not $q$-solvable. Then ${\rm{HG}}(G)<ek^2q^k$.
\end{corollary}

\begin{theorem}\label{trees}
  Let $G$ be a graph containing a vertex $v$ of degree one. If $q\ge3$ and $G$ is $q$-solvable, then $G\setminus\{v\}$ is also $q$-solvable.
\end{theorem}

As a simple corollary of Theorem \ref{trees} and the fact that any tree contains a degree-$1$ vertex, we get the following result which was originally proved in \cite{Butler08}, see Corollary 9 there.

\begin{corollary}[\cite{Butler08}]
\label{stamstam}
Trees are not $3$-solvable.
\end{corollary}

\subsection{Linear hat guessing numbers}

\noindent If $q$, the number of possible hat colors, is a prime power, then the guessing strategies of the vertices can be viewed as multivariate polynomials over the finite field $\mathbb{F}_q$, where we identify $Q$ with $\mathbb{F}_q$.
In such a scenario we would like to understand how the hat guessing number depends on the complexity of the multivariate polynomials being used.
In particular, how it is affected if one restricts the degree of the multivariate polynomials to be bounded from above by some constant.
The most simple and yet nontrivial case to be considered is when the multivariate polynomials are linear.
A guessing strategy of $G$ is called {\it linear} if the guessing function of each vertex is an affine function of the colors of its neighbors.
Furthermore, the  {\it linear hat guessing number} ${\rm{HG_{lin}}}(G)$ is the largest prime power $q$ for which the graph  $G$ is $q$-solvable by a linear guessing strategy. We say that  $G$ is {\it linearly $q$-solvable} if there exists a linear guessing strategy for $G$ over $\mathbb{F}_q$.

We are not aware of any paper which specifically considers linear guessing strategies for this hat guessing problem, therefore known results are scarce.
It is easy to verify that ${\rm{HG_{lin}}}(K_q)=q$ \cite{feige2004you}, and it is known that ${\rm{HG_{lin}}}(C_4)=3$ \cite{Gadouleau09,Szczechla17}.

It is worth noting that if $p<q$ are two integers, then $G$ is $q$-solvable implies that $G$ is also $p$-solvable.
However, this does not follows automatically in the case of linear solvability, i.e., if $G$ is linearly $q$-solvable it does not necessarily imply that it is also linearly $p$-solvable (assuming $p$ and $q$ are prime powers).
Clearly this follows if $q$ is a power of $p$, and therefore $\mathbb{F}_q$ contains $\mathbb{F}_p$ as a subfield, but this implication does not follow generally.
In fact it is of interest to construct a graph which is linearly $q$-solvable, but is not linearly $p$-solvable for $p<q$.

Since we are concerned with algebraic aspects of the hat guessing problem, it is with no surprise that we use algebraic methods to derive our results, most notably the Combinatorial Nullstellensatz \cite{AlonCombi}.
We present negative results (upper bonds) on the linear hat guessing numbers of several graph families, including cycles, complete bipartite graphs, degenerate graphs and graphs with bounded minimum rank, to be defined below.
These results are proved in Section \ref{sec:linear}.

Our first result on the linear solvability of graphs provides the exact value of the linear hat guessing number of cycles.
As already mentioned, it is known that ${\rm{HG_{lin}}}(K_3)= {\rm{HG_{lin}}}(C_4)= 3$, and we show that longer cycles are not linearly $3$-solvable.

\begin{theorem}\label{Linearforcycles}
	For any integer $n\ge 5$, ${\rm{HG}}_{lin}(C_n)\le 2.$
\end{theorem}

Combined with the result of Szczechla \cite{Szczechla17}, cycles are the first known examples of graphs for which non-linear guessing strategies outperform linear ones.
More precisely, $C_n$ is $3$-solvable if and only if $n=4$ or is a multiple of $3$, while for $n>4$ it is only linearly $2$-solvable.

The cycles provide a moderate separation between linear and non-linear guessing strategies, however a much more significant separation can be shown in the case of complete bipartite graphs.
Theorem \ref{polynomialbipartite} shows that for sufficiently large $n$, the hat guessing number of $K_{n,n}$ is at least $n^{\frac{1}{2}-o(1)}$. On the other hand, linear guessing strategies are significantly less powerful here, as described in the next theorem.

\begin{theorem}\label{linearK_n,n}
	$K_{n,n}$ is not linearly $q$-solvable for any proper prime power $q$.
\end{theorem}

The proof of Theorem \ref{linearK_n,n} relies on a result of Alon and Tarsi \cite{Alon89} (see Lemma \ref{alon89} below), which is only known to hold for finite fields of a \emph{proper} prime power order, but is conjectured in \cite{Alon89} to hold for any prime power $q\geq 4$.
If indeed the conjecture holds, then it would imply that Theorem \ref{linearK_n,n} holds for any prime power $q\geq 4$.
Note that this cannot be further improved, since $C_4 = K_{2,2}$ is linearly $3$-solvable.
For more details, see Subsection \ref{subsec:linear_bipartite}.

The next two results relate linear solvability to other graph parameters, namely, degeneracy and minimum rank.
In Question \ref{degree} we ask whether the hat guessing number is bounded from above by the degeneracy of the graph.
Here we resolve this question for the linear solvability.
Notice that the bound below is tight for $K_q$ since it is $(q-1)$-degenerate and ${\rm{HG_{lin}}}(K_q)=q$.

\begin{theorem}\label{degeneratenonlinear}
	For any $d$-degenerate graph $G$, ${\rm{HG_{lin}}}(G)\le d+1$.
\end{theorem}

An $n\times n$ matrix $M$ over $\mathbb{F}_q$ {\it fits} a graph $G$ with $n$ vertices if $M_{i,i}\neq 0$ for $1\le i\le n$ and $M_{i,j}=M_{j,i}=0$ if $v_i$ and $v_j$ are not connected in $G$, i.e., $\{v_i,v_j\}\not\in E(G)$.
The {\it minimum rank} of a graph $G$ denoted by ${\rm{mr}}(G)$, is defined to be the minimum integer $r$ for which there exists a matrix $M$ over some finite field which fits $G$ and $\rank(M)=r$.
This parameter of a graph was initially introduced by Haemers \cite{Haemers79,Haemers78} as an upper bound for the Shannon capacity of a graph \cite{Shannon}.
Notice that the minimum rank was originally defined over any field, not necessarily finite, however we use the above definition. The last result connects linear solvability and minimum rank, where loosely speaking, it claims that large minimum rank implies small linear solvability.

\begin{theorem}\label{minrk1}
	Let $G$ be a graph with $n$ vertices, then ${\rm{HG_{lin}}}(G)\le n-{\rm{mr}}(G)+1$.
\end{theorem}

Theorem \ref{minrk1} is tight for any prime power $q$, as ${\rm{mr}}(K_q)=1$ and ${\rm{HG_{lin}}}(K_q)=q$.  Notice that in general it is not
practical to apply the above bound since there is no known
efficient algorithm which computes the minimum rank of a graph over
a given field and
in fact this problem is known to be NP-hard \cite{Peeters}.

\subsection{Organization}

\noindent The rest of the paper is organized as follows. In Section \ref{sec:related} we briefly review other versions of hat guessing games.
Necessary definitions and notations are given in Section \ref{sec:prelims}.
Theorems \ref{polynomialbipartite} and \ref{polynomialdirected} are proved in Section \ref{sec:multipartite}.
In Section \ref{sec:bounded_deg} we present the proofs of Theorems \ref{maxdeg} and \ref{onemoreobservation}.
In Section \ref{sec:linear} we consider linear hat guessing numbers and prove  Theorems
\ref{Linearforcycles}, \ref{linearK_n,n}, \ref{degeneratenonlinear} and \ref{minrk1}.

\section{Related work}
\label{sec:related}
\noindent In the literature, there are several versions of hat guessing games.
The most famous (maybe also the earliest) version was introduced by Ebert \cite{ebert1998applications} and advertised by Robinson \cite{newyorktimes} in New York Times as a recreational mathematical game.
In that version the graph $G$ is a clique, each player gets either a red or a blue hat with equal probability, and is asked to guess or just pass.
The players win if at least one player guesses correctly and no one guesses wrong, otherwise they lose.
The goal is to design a guessing strategy maximizing the probability of winning.
Krzywkowski \cite{krzywkowski2010modified} considered a variation of the above game, where the players are allowed to guess sequentially.
Winkler \cite{winkler2002games} investigated another modification in which the players cannot pass and the objective is to guarantee as many correct guesses as possible, assuming the worst-case hat coloring.
Feige \cite{feige2004you} and Aggarwal et al. \cite{aggarwal2005derandomization} generalized Winkler's problem to $q$ colors.
The version considered in this work was introduced by Butler et al. \cite{Butler08}, where instead of being a clique, the sight graph can be an arbitrary (directed) graph.
This version has been investigated further by Gadouleau and Georgiou \cite{Gadouleau09}, Szczechla \cite{Szczechla17} and Gadouleau \cite{gadouleau2018finite}.
The reader is referred to the theses of Farnik \cite{Farnik2015} and Krzywkowski \cite{krzywkowski2012hat} for extensive reviews on different hat guessing games.

Hat guessing problems have attracted increasing attention and found many applications and connections to several seemingly unrelated research areas, such as coding theory \cite{Ebertcodingtheory}, auctions \cite{aggarwal2005derandomization}, network coding \cite{riis2007graph,gadgraph} and finite dynamical systems \cite{gadouleau2018finite}.

The phenomena, exhibited in Theorem \ref{Linearforcycles} and Theorem \ref{linearK_n,n}, that nonlinear encoding functions (operations) can outperform linear ones is known to exist in information theory and communication complexity, however explicit examples for it are rather sporadic.
We list two of them.
In the context of index coding \cite{birk}, Lubetzky and Stav \cite{stav} showed that nonlinear encoding functions can significantly outperform the optimal linear encoding function.
Recently, Alon, Efremenko and Sudakov \cite{klim} showed that testing equality among bit strings over communication graphs can be done much more efficiently using nonlinear functions than by restricting to linear ones.

\section{Preliminaries}
\label{sec:prelims}

\noindent Throughout the paper $G$ is a graph with the vertex set $V(G)=\{v_1,\ldots,v_n\}$ and the edge set $E(G)$.
For a set $S\s V(G)$, the {\it induced subgraph} of $G$ on $S$ is
the graph with vertex set $S$ and all edges of $E(G)$ with
both endpoints in $S$.
$G$ is said to be {\it $d$-degenerate} if there exists an ordering $v_{i_1},\ldots,v_{i_n}$ of $V(G)$, such that for $2\le j\le n$, $v_{i_j}$ is connected to at most $d$ vertices among $v_{i_1},\ldots,v_{i_{j-1}}$.
For two distinct vertices $u,v\in V(G)$, the {\it distance} between $u,v$ is the length of a shortest path connecting them and $\infty$ if there does not exist such a path.
The {\it girth} of a directed graph is the length of its the shortest directed cycle and $\infty$ if it does not contain any directed cycle.
A directed graph is called \emph{oriented} if its girth is at least 3.

For a positive integer $q$ let $[q]=\{1,\ldots,q\}$, and we view the set of possible colors as $Q=[q]$.
When we consider linear guessing strategies, we assume that $q$ is a prime power and $Q$ is the finite field $\mathbb{F}_q$.

The set of colors assigned to the vertices of $G$ is represented by a vector $x$ of length $n$, $x=(x_1,\ldots,x_n)\in[q]^n$, where $x_i$ is the  color assigned to $v_i$.
For a subset $S\s V(G)$ of vertices, $\chi(S)$  denotes the colors assigned to the vertices in $S$; that is,
$$\chi(S)=x|_S=(x_i:v_i\in S)\in[q]^{|S|},$$ is the restriction of $x$ to coordinates $i$ with $v_i\in S$.
Moreover, if the coloring $x\in [q]^n$ and the guessing strategies are understood from the context, we write $\phi(S)=1$ if at least one of the vertices in $S$ guesses correctly, otherwise we write $\phi(S)=0$.
The notation $\chi(\cdot)$ and $\phi(\cdot)$ are frequently used to simplify the proofs.

The guessing strategy of $v_i$ is a function $f_i$, whose variables are the colors assigned to the neighbors of $v_i$.
Consequently, we can write $f_i=f_i(x)=f_i(x_{i_1},\ldots,x_{i_{d_i}})$,
where $v_{i_1},\ldots,v_{i_{d_i}}$ are the neighbors of $v_i$.
One can easily see that a vertex $v_i$ guesses its color correctly if and only if $x_i-f_i=0$, and $f:=(f_1,\ldots,f_n)$ forms a proper guessing strategy for $G$ if and only if the function

\begin{equation}\label{guessingstrategy}
  F(x)=\prod_{i=1}^n (x_i-f_i)
\end{equation}

\noindent vanishes on $Q^n$, where for $Q=[q]$, $f_i$ is an $\mathbb{R}$-valued function with $d_i$ variables, and for a prime power $q$ and $Q=\mathbb{F}_q$, $f_i$ is an $\mathbb{F}_q$-valued function with $d_i$ variables.

\subsection{The Hamming ball condition for bipartite graphs}

\noindent For a positive integer $m$ and $x,y\in[q]^m$, the {\it Hamming distance} between $x,y$ is the number of coordinates in which they differ, i.e., $${\rm{d}}(x,y)=|\{i\in[m]:x_i\neq y_i\}|.$$
For a vector $a\in[q]^m$ and an integer $r$, the {\it Hamming ball} $B_r(a)$ of radius $r$ and center $a$ is the set of vectors of $[q]^m$ that are of  Hamming distance at most $r$ from $a$, i.e., $$B_r(a)=\{x\in[q]^m:{\rm{d}}(x,a)\le r\}.$$
In the proofs we will make use of the  following simple but useful observation.
For any $a\in[q]^m$ and  $x\in B_{m-1}(a)$, $x$ and $a$ agree on at least one coordinate, i.e., there exists an $i$ such that $a_i=x_i$.

Next we design a guessing strategy for complete bipartite graphs based on this observation.
Consider the complete bipartite graph $K_{m,n}$ with left part $V_L=\{u_1,\ldots,u_m\}$ and right part $V_R=\{v_1,\ldots,v_n\}$.
The following lemma, which is a special case of a more general result introduced in \cite{Gadouleau09}, provides a sufficient condition for the solvability of complete bipartite graphs.

\begin{lemma}[The Hamming ball condition, see Section 2, \cite{Gadouleau09}]\label{characterization}
$K_{m,n}$ is $q$-solvable if there is a guessing strategy for the vertices of $V_R$ which satisfies the following property.
For any coloring $\chi(V_R)=y\in[q]^n$ of the right part, the set

$$\ma{C}_y:=\{x\in[q]^m:\text{ given}~\chi(V_R)=y,~\text{if}~\chi(V_L)=x~\text{then}~\phi(V_R)=0\},$$

\noindent is contained in a Hamming ball $B_{m-1}(a^y)$ for some $a^y\in[q]^m$.
\end{lemma}

\begin{proof}
    Assume that there exists a guessing strategy for $V_R$ which satisfies the above property.
    Next, we design a guessing strategy for $V_L$.
    If $\chi(V_R)=y$, then $u_i$ (which is the $i$th vertex of $V_L$) guesses its color to be $a^y_i$.
    If $\phi(V_R)=1$ we are done, otherwise $\phi(V_R)=0$ and  $\chi(V_L)\in\ma{C}_y\s B_{m-1}(a^y)$.
    Therefore at least one vertex in $V_L$ must guess correctly as ${\rm{d}}(\chi(V_L),a^y)\le m-1$.
\end{proof}

In order to apply Lemma \ref{characterization} we should identify sufficient conditions for which a set of vectors of $[q]^m$ is contained in a Hamming ball of radius $m-1$.
Intuitively, if the set is small enough then it is clear that it should be contained in a Hamming ball of small radius.
For example, any set $x^1,\ldots,x^m$ of $m$ vectors in $[q]^m$ is contained in a Hamming ball of radius $m-1$, since $x^i\in B_{m-1}\big((x^1_1,\ldots,x^m_m)\big)$ for every $1\le i\le m$, where $x^i_i$ is the $i$th coordinate of $x^i$.
This useful observation is stated as the following lemma.

\begin{lemma}\label{hammingballtrivial}
Any set of at most $m$ vectors in $[q]^m$ is contained in a Hamming ball of radius $m-1$.
\end{lemma}

One can apply a probabilistic argument to improve the result of Lemma \ref{hammingballtrivial} for $m\gg q$.

\begin{lemma}\label{hammingball}
  Any set $\ma{C}\s[q]^m$ of at most $e^{m/q}$ vectors is contained in a Hamming ball of radius $m-1$.
\end{lemma}

\begin{proof}
  Pick $a\in[q]^m$ to be the center of the Hamming ball uniformly at random.
  Since for any $x\in\ma{C}$, $\Pr[{\rm{d}}(a,x)=m]=(1-\fr{1}{q})^m$, we get by the union bound
  $$\Pr[\exists~x\in\ma{C},~s.t.~{\rm{d}}(a,x)=m]\le|\ma{C}|(1-\fr{1}{q})^m<|\ma{C}|e^{-\fr{m}{q}}\le 1,$$
  which implies that with positive probability there exists a vector $a\in[q]^m$ such that $\ma{C}\s B_{m-1}(a)$.
\end{proof}

The main idea in the proofs of Theorem \ref{polynomialbipartite} and Theorem \ref{polynomialdirected} can be explained briefly as follows.
Consider the bipartite graph $K_{m,n}$, where our goal is to show that ${\rm{HG}}(K_{m,n})\ge q$.
Let the vertices of $V_R$ pick their guessing strategies $f_{v_i}:[q]^n\rightarrow [q], 1\le i\le n$ independently and uniformly at random.
For a fixed coloring $\chi(V_R)=y\in[q]^n$ of the right part, the probability that no vertex of $V_R$ guesses correctly given the left part has coloring $\chi(V_L)=x$, satisfies

$$\Pr[\phi(V_R)=0]=\Pr[{\rm{d}}\Big(y,\big(f_{v_1}(x),\ldots,f_{v_n}(x)\big)\Big)=n]=(1-\fr{1}{q})^n.$$

\noindent Then by linearity of expectation

$${\rm{E}}[|\ma{C}_y|]=q^m(1-\fr{1}{q})^n<q^me^{-\fr{n}{q}}\le 1$$

\noindent for $n\ge mq\ln q$, where $\ma{C}_y$ is defined in Lemma \ref{characterization}.
Therefore, if $n$ is large enough then with high probability $|\ma{C}_y|$ is small enough and by Lemma \ref{hammingballtrivial} or Lemma \ref{hammingball} it is contained in a Hamming ball of radius $m-1$.
Finally, by Lemma \ref{characterization} the graph is $q$-solvable.

In order to make sure that $|\ma{C}_y|$ is small for any coloring $\chi(V_R)=y$, next we introduce the concept of saturated matrices.

\subsection{Matrices and guessing strategies}

\noindent Given a guessing strategy for the vertices of $V_R$, we represent it by a matrix as follows.
Let $A$ be an $n\times q^m$ $q$-ary matrix, whose rows and columns are indexed by the $n$ vertices of $V_R$ and the $q^m$ different possible colorings of $V_L$, respectively.
$A_{i,x}$, the entry of $A$ in row $v_i\in V_{R}$ and column $x\in[q]^m$, is the guess of $v_i$ given the coloring of the left part is $\chi(V_L)=x$ .

Next we introduce a matrix property which is sufficient for the condition of Lemma \ref{characterization} to hold.
An $n\times l$ $q$-ary matrix $M$ is called {\it $t$-saturated} if for any set $T\subseteq [l]$ of $t$ columns, there exists at least one row $r$ such that the restriction of $M$ to row $r$ and columns in $T$ is onto $[q]$, i.e.,
$$\{M_{r,i}:i\in T\}=[q].$$
Notice that $t$-saturated matrices exist only if $t\ge q$.
In the literature, the rows of a $t$-saturated matrix is known as a {\it $(l,q,t)$}-family \cite{Chakrabortylistdecoding} and in particular, for $t=q$ such a family is called a {\it $t$-perfect hash family} \cite{Fredman84}.
Since in this paper the guessing strategies are frequently represented in the matrix form, for convenience we simply view these families as matrices.
In information theory the columns of a $t$-saturated matrix were used to construct the {\it zero-error list-decoding code for the $q/(q-1)$ channel with list size $t-1$} \cite{Chakrabortylistdecoding,Elias88}.
The construction of saturated matrices and their applications were studied extensively (see e.g., \cite{bhandari2018bounds,Chakrabortylistdecoding,Fredman84,Korner86}).

The following lemma shows the existence of saturated matrices.
In the next section these matrices will be used to construct guessing strategies for complete bipartite graphs.

\begin{lemma}[\cite{Chakrabortylistdecoding,Fredman84}]\label{saturated}
Let $n,l,q$ be integers, then
  \begin{itemize}
    \item [$(i)$] an $n\times l$ $q$-ary $q$-saturated  matrix ($q$-perfect hash family) exists for $$l\le 2^\frac{-q}{q-1}(\fr{1}{1-\fr{q!}{q^q}})^{\fr{n}{q-1}}.$$
    \item [$(ii)$] for $t\ge q\ln q$, an $n\times l$ $q$-ary $t$-saturated matrix ($(l,q,t)$-family) exists for $$l\le 2^\frac{-t}{t-1}(\fr{1}{q}e^{\fr{t}{q}})^{\fr{n}{t-1}}.$$
  \end{itemize}
\end{lemma}

Notice that Lemma \ref{saturated} $(i)$ was originally proved in \cite{Fredman84} and Lemma \ref{saturated} $(ii)$ was stated in Proposition 2 of \cite{Chakrabortylistdecoding}, but with no proof.
For completeness we present the proof of this lemma in Appendix \ref{sec:proof_lem_saturated}.

\section{Complete multipartite graphs and digraphs}
\label{sec:multipartite}

\noindent In this section we prove Theorem \ref{polynomialbipartite} and Theorem \ref{polynomialdirected}.
We first prove the results for bipartite graphs, and then generalize them to multipartite graphs.

\subsection{Complete bipartite graphs}

\noindent Consider the complete bipartite graph $K_{m,n}$ with vertex parts $V_L=\{u_1,\ldots,u_m\}$ and $V_R=\{v_1,\ldots,v_n\}$.
Our goal is to design a guessing strategy for the vertices of $V_R$ which satisfies the condition of Lemma \ref{characterization}.
The following lemma shows that $t$-saturated matrices are sufficient for this purpose.

\begin{lemma}\label{completebipartite}
  If there exists an $n\times q^m$ $q$-ary $(m+1)$-saturated matrix $A$, then the complete bipartite graph $K_{m,n}$ is $q$-solvable.
\end{lemma}

\begin{proof}
Given an $n\times q^m$ $q$-ary $(m+1)$-saturated matrix $A$, index its rows and columns by the vertices of $V_R$ and the $q^m$ possible colorings of $V_L$, respectively.
The guess of $v_i$ given coloring $\chi(V_L)=x$ is $A_{i,x}$.
Next, recall that for a coloring $\chi(V_R)=y$ we define $$\ma{C}_y:=\{x\in[q]^m:\text{ given}~\chi(V_R)=y,~\text{if}~\chi(V_L)=x~\text{then}~\phi(V_R)=0\}.$$
We have the following claim.

\begin{claim}\label{core}
  For any coloring $y\in[q]^n$ assigned to $V_R$, $|\ma{C}_y|\le m$.
\end{claim}

Assume the claim is correct, then by Lemma \ref{hammingballtrivial} $\ma{C}_y$ is contained in a Hamming ball of radius $m-1$ centered at some $a^y\in[q]^m$, and by Lemma \ref{characterization} $K_{m,n}$ is $q$-solvable.

\vspace{5pt}

\noindent{\it Proof of Claim \ref{core}.} Assume to the contrary that there exists $y\in[q]^n$ with $|\ma{C}_y|\ge m+1$.
Since $A$ is $(m+1)$-saturated, there exists a row $i$ of $A$, indexed by a vertex $v_i\in V_R$, such that $\{A_{i,x}:x\in \ma{C}_y\}=[q]$. Therefore, vertex $v_i$ guesses its color correctly for at least one of the colorings in $\ma{C}_y$, which contradicts the definition of $\ma{C}_y$.
\end{proof}

By combining Lemma \ref{saturated} and Lemma \ref{completebipartite} we have the following proposition.

\begin{proposition}\label{explicitcompletebipartite}

  \begin{itemize}
    \item [$(i)$] There exists a constant $c_1$ not depending on $q$ such that $${\rm{HG}}(K_{q-1,c_1e^qq^{1.5}\ln q})\ge q.$$

    \item [$(ii)$] There exists a constant $c_2$ not depending on $q$ such that
    $${\rm{HG}}(K_{2q\ln q,c_2q^2(\ln q)^2})\ge q.$$
  \end{itemize}
\end{proposition}

\begin{proof}
  The first statement applies  Lemma \ref{saturated} $(i)$ and Lemma \ref{completebipartite} with $n=c_1e^qq^{1.5}\log q$, $l=q^{q-1}$ and $m=q-1$ for some constant $c_1$.
  For the second statement one applies Lemma \ref{saturated} $(ii)$ and Lemma \ref{completebipartite} with  $n=c_2q^2(\ln q)^2$, $l=q^m,m=2q\ln q$ and $t=m+1$ for some constant $c_2$.
  We omit the computations. 
\end{proof}

For large enough $q$, Proposition \ref{explicitcompletebipartite} $(i)$ improves on the known result ${\rm{HG}}(K_{q-1,(q-1)^{q-1}})\ge q$ \cite{Gadouleau09}.
For large enough $n$, Proposition \ref{explicitcompletebipartite} $(ii)$ implies that asymptotically ${\rm{HG}}(K_{n,n})\ge n^{\fr{1}{2}-o(1)}$, which provides a positive answer to Question \ref{question1}. As a natural extension of this result, it is interesting to study the following question.

\begin{question}
Proposition \ref{explicitcompletebipartite} $(ii)$ implies that for $f(q)=\Theta(q\ln q)$ we have
${\rm{HG}}(K_{f(q), q^{2+o(1)}}) \geq q$.
Can this be improved? Specifically, can one show that for some function $g(q) = o(q \ln q)$ and constant $c \geq 1$, ${\rm{HG}}(K_{g(q), q^{c}}) \geq q$? Does this hold for $g(q) = \Theta(q)$?
\end{question}

\subsection{Complete multipartite graphs}

\noindent Similarly to Proposition \ref{explicitcompletebipartite}, Theorem \ref{polynomialbipartite} is proved by applying  Lemma \ref{saturated} and Lemma \ref{completebipartite} together with some additional combinatorial ideas.
We say that a graph $G$ on $n$ vertices is {\it partially $q$-solvable} with respect to a coloring set $\ma{C}\s[q]^n$ if there is a guessing strategy such that at least one vertex of $G$ guesses correctly, given the coloring is restricted to be one of the colorings in $\ma{C}$.
For example, $G$ is $q$-solvable if and only if it is partially $q$-solvable with respect to $\ma{C}=[q]^n$.

Given a set $\ma{C}\s[q]^{rm}$ and $y\in [q]^m$ we define ${\rm{Suffix}}(y)=\{x\in[q]^{(r-1)m}:~(x,y)\in\ma{C}\}$.
Using this notation with have the following lemma.

\begin{lemma}\label{partiallyq-solvable}
  For integers $q,m,r\ge 1$, the complete $r$-partite graph $K_{m,\ldots,m}$ with $m$ vertices in each part, is partially $q$-solvable with respect to any coloring set $\ma{C}\s[q]^{rm}$ of size at most $m^r$.
\end{lemma}

\begin{proof}
  We apply induction on $r$.
  For $r=1$, we have the empty graph on $m$ vertices $v_1,\ldots,v_m$.
  Since $\ma{C}$ is of size at most $m$ it is contained in a Hamming ball of radius $m-1$, say, $B_{m-1}(a)$ for some $a\in [q]^m$.
  Clearly, if $v_i$ guesses $a_i$ for $1\le i\le m$, then at least one vertex guesses correctly, and the graph is partially $q$-solvable with respect to $\ma{C}$.
			
Next assume the statement is correct for $r-1$ and we  prove it for $r$.
Consider the complete $r$-partite graph $K_{m,\ldots,m}$ and let $\ma{C}\s[q]^{rm}$ be a set of colorings of size at most $m^{r}$.
Let $V_R$ be one of the $r$ parts of vertices, and let $V_L$ be the remaining $r-1$ parts.
Since $|\ma{C}|\le m^r$, there are at most $m$ $y$'s in $ [q]^m$  with $|{\rm{Suffix}}(y)|> m^{r-1}$. Denote this set of $y$'s by  $\ma{Y}$.
Clearly, since $|\ma{Y}|\leq m$ it is contained in a Hamming ball of radius $m-1$ centered at some $a$.
Notice that for $y\in [q]^m\backslash\ma{Y}$, $|{\rm{Suffix}}(y)|\le m^{r-1}$, and that the induced subgraph on $V_L$ is a complete $(r-1)$-partite graph.
Thus by the induction hypothesis it is partially $q$-solvable with respect to any ${\rm{Suffix}}(y)$ with $y\in [q]^m\backslash\ma{Y}$.
Next we define the guessing strategy for the $r$-partite graph.
For $1\le i\le m$, the $i$th vertex of $V_R$ guesses its color to be $a_i$.
The vertices of $V_L$ use the following strategy.
If $V_R$ is colored by a coloring $y\in \ma{Y}$, then each vertex guesses arbitrarily.
Otherwise, if $y\notin \ma{Y}$ then the vertices of $V_L$ use the strategy for the $(r-1)$-partite graph with respect to the colorings ${\rm{Suffix}}(y)$.
It is easy to verify that such a strategy ensures that at least one vertex guesses correctly, and the result follows.
\end{proof}

Combining the ideas introduced in the proofs of Lemma \ref{completebipartite} and Lemma \ref{partiallyq-solvable}, we have the following result.

\begin{proposition}\label{completemultipartite}
  If there exists an $n\times q^{(r-1)m}$ $q$-ary $t$-saturated matrix $A$ with $t\le m^{r-1}$, then the complete $r$-partite graph $K_{m,\ldots,m,n}$, in which there are $r-1$ vertex parts of size $m$ and one vertex part with size $n$, is $q$-solvable.
\end{proposition}

\begin{proof}
  Split the vertices into two parts, say, $V_L$ and $V_R$, such that $V_L$ consists of the $r-1$ vertex parts of size $m$ and $V_R$ is the remaining part of size $n$.
  Index the rows and columns of $A$ by the vertices in $V_R$ and the $q^{(r-1)m}$ possible colorings of $V_L$, respectively.
  As before, we let the vertices in $V_R$ guess according to the $t$-saturated matrix $A$, i.e., given the coloring $\chi(V_L)=x$, the $i$th vertex of $V_R$ guesses $A_{i,x}$.
  Similarly to Claim \ref{core}, the property of $t$-saturated matrices guarantees that for any coloring $y$ of $V_R$, the set
  $$\ma{C}_y:=\{x\in[q]^{(r-1)m}:\text{ given}~\chi(V_R)=y,\text{ if }\chi(V_L)=x \text{ then }\phi(V_R)=0\}$$
  is of size at most $t-1$.
  It remains to design an appropriate guessing strategy for the vertices in $V_L$. Observe that since $t\le m^{r-1}$, Lemma \ref{partiallyq-solvable} implies that $K_{V_L}$, the induced subgraph on $V_L$ (which is a complete $(r-1)$-partite graph) is partially $q$-solvable with respect to $\ma{C}_y$.
  Thus the vertices of $V_L$ upon seeing the coloring $y$ of $V_R$, they guess according to a strategy for which $K_{V_L}$ is partially $q$-solvable with respect to $\ma{C}_y\s[q]^{(r-1)m}$.
  By combining the two guessing strategies for $V_L$ and $V_R$, one can conclude that at least one vertex must guess its color correctly.
\end{proof}

  Next we prove  Theorem \ref{polynomialbipartite}.

  \begin{proof}[\textbf{Proof of Theorem \ref{polynomialbipartite}}]
    By applying Lemma \ref{saturated} $(ii)$ with $n=c(q\ln q)^{\fr{r}{r-1}}$, $l=q^{(r-1)m},m=(2q\ln q)^{\fr{1}{r-1}}$ and $t=2q\ln q$ for some constant $c$, one obtains an $n\times q^{(r-1)m}$ $t$-saturated matrix with $t=m^{r-1}$.
    It thus follows from Proposition \ref{completemultipartite} that ${\rm{HG}}(K_{m,\ldots,m,n})\ge q$.
  \end{proof}

\subsection{Complete multipartite directed cycles}

\noindent In this section we prove Theorem \ref{polynomialdirected}.
The proof relies also on saturated matrices.
Recall that in the proof for (undirected) complete bipartite graphs, one part of vertices guesses according to a saturated matrix, whereas the second part guesses according to the Hamming ball condition.
The proof for complete $r$-partite directed cycles follows along the same lines as follows.
Each of the first $r-1$ vertex parts guesses according to a different saturated matrix, and the $r$th part guesses according to the Hamming ball condition.

Recall the definition of $\vec{C}_{n_1,\ldots,n_r}$ introduced in Subsection 1.1.
Let us first present the proof of Theorem \ref{polynomialdirected} for $r=3$.

\begin{proposition}\label{complete3directedcycle}
  If there exist an $n_1\times q^{n_2}$ $q$-ary $t_1$-saturated matrix $A_1$ and an $n_2\times q^{n_3}$ $t_2$-saturated matrix $A_2$ such that $t_1t_2\le e^{n_3/q}$,
  then the complete 3-partite directed cycle $\vec{C}_{n_1,n_2,n_3}$ is $q$-solvable.
\end{proposition}
\begin{proof}

  For $i=1,2,3$ let $V_i$ be the $i$th vertex part of $\vec{C}_{n_1,n_2,n_3}$ of size $n_i$.
  Denote the colors assigned to the vertex parts  $V_1,V_2,V_3$ by vectors $x\in[q]^{n_1},y\in[q]^{n_2},z\in[q]^{n_3}$, respectively.
  The vertices in $V_1$ and $V_2$ guess theirs colors according to the saturated matrices $A_1$ and $A_2$ respectively, as described in the proof of Lemma \ref{completebipartite}.
  Next we describe the guessing strategy for $V_3$.
  We will show that given the coloring of $V_1$ and $V_2$ there are at most $t_1t_2$ colorings of $V_3$ for which all the vertices of $V_1$ and $V_2$ guess incorrectly.
  Since by assumption $t_1t_2\le e^{n_3/q}$ then by Lemma \ref{hammingball} these vectors are contained in a Hamming ball of radius $n_3-1$,  centered at some $a$.
  Clearly, if upon seeing the coloring of $V_1$, the $i$th vertex of  $V_3$ guesses  $a_i$, i.e., the vertices in $V_3$ collectively guess the vector $a$, then at least one vertex of $V_3$ makes a correct guess and the result follows.

  For $x\in [q]^{n_1}$, define
	$$\ma{C}_x:=\{y\in[q]^{n_2}:\text{ given}~\chi(V_1)=x,\text{ if}~\chi(V_2)=y \text{ then } \phi(V_1)=0\},$$
  and for  $y\in [q]^{n_2}$, define
  $$\ma{C}_y:=\{z\in[q]^{n_3}:\text{ given}~\chi(V_2)=y,\text{ if}~\chi(V_3)=z \text{ then }\phi(V_2)=0\}.$$
  Suppose that for some colorings $\chi(V_1)=x,\chi(V_2)=y, \chi(V_3)=z$, all the vertices of $V_1\cup V_2$ guess incorrectly, then by definition $y\in\ma{C}_x$ and $z\in\ma{C}_y$. 	
  Hence we have
  $$z\in\cup_{y\in\ma{C}_x}\ma{C}_y.$$
  Since $A_1$ and $A_2$ are $t_1$ and $t_2$-saturated matrices, respectively, then $|\ma{C}_x|< t_1$ for $x\in [q]^{n_1}$, $|\ma{C}_y|< t_2$ for $y\in [q]^{n_2}$ and $|\cup_{y\in\ma{C}_x}\ma{C}_y|<t_1t_2$, which proves the result.
\end{proof}

The following more general result is proved using a similar argument.

\begin{proposition}\label{completemultidirectedcycle}
  If there exist $r-1$ matrices $A_1,\ldots,A_{r-1}$ such that
  \begin{itemize}
    \item [$(i)$] for each $1\le i\le r-1$, $A_i$ is an $n_i\times q^{n_{i+1}}$ $q$-ary $t_i$-saturated matrix;
    \item [$(ii)$] $\prod_{i=1}^{r-1} t_i\le e^{n_r/q}$.
  \end{itemize}
  Then the complete $r$-partite directed cycle $\vec{C}_{n_1,\ldots,n_r}$ is $q$-solvable.
\end{proposition}

\begin{proof}
  The proof is very similar to that of Proposition \ref{complete3directedcycle}, and is roughly sketched as follows.
  For  $1\le i\le r$, let $x^i\in[q]^{n_i}$ be the colors assigned to the vertices of $V_i$.
  For $1\le i\le r-1$ the vertices in $V_i$ guess their colors according to $A_i$, as described in the proof of Lemma \ref{completebipartite}.
  Similar to the proof of Proposition \ref{complete3directedcycle}, it suffices to show that for any given set of colorings of $V_1,\ldots,V_{r-1}$, the set of colorings of $V_r$ for which all the vertices in $\cup_{i=1}^{r-1} V_i$ guess incorrectly, must be contained in some Hamming ball of radius $n_r-1$ in $[q]^{n_r}$.

  To this end, for $1\le i\le r-1$ and $x^i\in[q]^{n_i}$, define
  $$\ma{C}_{x^i}:=\{y\in[q]^{n_{i+1}}:\text{ given}~\chi(V_i)=x^i,\text{ if }\chi(V_{i+1})=y \text{ then }\phi(V_i)=0\}.$$

  \noindent Suppose that there exist colorings $\chi(V_i)=x^i$ for $1\le i\le r$ such that all the vertices in $\cup_{i=1}^{r-1}V_i$ guess incorrectly.
  By the definition of $\ma{C}_{x^i}$ it holds that
  for each $1\le i\le r-1$, $x^{i+1}\in\ma{C}_{x^i}$, which implies that

  \begin{equation}\label{stam}
      x^r\in \cup_{y^2\in\ma{C}_{x^1}}\cup_{y^3\in\ma{C}_{y^2}}\cdots\cup_{y^{r-1}\in\ma{C}_{y^{r-2}}}\ma{C}_{y^{r-1}},
  \end{equation}

\noindent where $x^1\in[q]^{n_1},y^2\in[q]^{n_2},\ldots,y^{r-1}\in[q]^{r-1}$.
Since for $1\le i\le r-1$, $A_i$ is $t_i$-saturated, we have that
$$|\ma{C}_{x^1}|< t_{1},~|\ma{C}_{y^2}|< t_{2},~\ldots,~|\ma{C}_{y^{r-1}}|< t_{r-1},$$
and that the size of the union in \eqref{stam} is at most $\prod_{i=1}^{r-1}t_i< e^{n_r/q}$.
Thus, by Lemma \ref{hammingball} these vectors are contained in some Hamming ball of radius $n_r-1$ in $[q]^{n_r}$, as required.
\end{proof}

Next we present the proof of Theorem \ref{polynomialdirected}.

\begin{proof}[\textbf{Proof of Theorem \ref{polynomialdirected}}]
We construct saturated matrices $A_1,\ldots,A_{r-1}$ as required in Proposition \ref{completemultidirectedcycle} by invoking Lemma \ref{saturated} $(ii)$ as follows.
For $1\le i\le r-1$, set $t_i=2q\ln q$, then there exists an $n_i\times q^{n_{i+1}}$ $q$-ary $t_i$-saturated matrix with $n_i= n_{i+1}(4q\ln q).$
It remains to pick the value of $n_r$ satisfying Proposition \ref{completemultidirectedcycle} $(ii)$.
Since $\prod_{i=1}^{r-1}t_i=(2q\ln q)^{r-1}=e^{(r-1)\ln(2q\ln q)}$, one can set $n_r=(r-1)q\ln(2q\ln q)$.

To conclude, for $t_1=\cdots=t_{r-1}=2q\ln q$ and
$$n_i=n_r(4q\ln q)^{r-i}=(r-1)\ln(2q\ln q)(4\ln q)^{r-i}q^{r+1-i}$$ for $1\le i\le r$, there exist matrices $A_1,\ldots,A_{r-1}$ satisfying both conditions of Lemma \ref{completemultidirectedcycle}.
Consequently, it holds that ${\rm{HG}}(\vec{C}_{n_1,\ldots,n_r})\ge q$.
\end{proof}

\section{Hat guessing number and other graph parameters}
\label{sec:bounded_deg}

\noindent In this section, we  present the proofs of Theorem \ref{maxdeg} and Theorem \ref{onemoreobservation}, whereas the proof of Theorem \ref{trees} is postponed to Appendix \ref{sec:proof_thm_trees}.
Although Theorem \ref{maxdeg} is a folklore \cite{Farnik2015}, we include its proof for completeness.
Both proofs make use of the Lov\'asz Local Lemma \cite{locallemma}, which states the following.

\begin{lemma}[Lov\'asz Local Lemma, Shearer's version \cite{Shearerversion}]\label{shearer}
  Let $A_1,\ldots,A_n$ be events in an arbitrary probability space. Suppose that each event $A_i$ is mutually independent of a set of all other events $A_j$
  but at most $d$, and that $\Pr[A_i]\le p$ for all $1\le i\le n$. If $epd\le1$, then $\Pr[\bigwedge_{i=1}^n\overline{A_i}]>0$.
\end{lemma}


\begin{proof}[\textbf{Proof of Theorem \ref{maxdeg}}]
  Assume $G$ has $n$ vertices $v_1,\ldots,v_n$ and let us assign colors of $[q]$ to each vertex in $G$ independently and uniformly at random.
  Given a guessing strategy, let $A_i, 1\le i\le n$  be the event that $v_i$ guesses correctly.
  To prove the theorem it suffices to show that $\Pr[\bigwedge_{i=1}^n\overline{A_i}]>0$ for $q\ge e\Delta$.
  It is easy to verify that for any guessing strategy $\Pr[A_i]=\fr{1}{q}$.
  For each $i\in [n]$, let $N_i$ be the set of neighbors of $v_i$ in $G$.
  Next, we have the following claim.

 \begin{claim}\label{lovasz1}
   For each $i\in [n]$, $A_i$ is mutually independent of $\{A_j:j\in[n]\setminus\{i\},~v_j\not\in N_i\}$.
 \end{claim}

  The proof of the theorem follows easily from the above claim as follows.
  Since $G$ has maximum degree $\Delta$, by Claim \ref{lovasz1} any event $A_i$ is mutually independent of the $A_j$'s $j\neq i$ but at most $\Delta$ of them.
  Therefore, by Lemma \ref{shearer} $\Pr[\bigwedge_{i=1}^n\overline{A_i}]>0$ for $q\ge e\Delta$; that is, with positive  probability no vertex guesses correctly.

 \vspace{5pt}

 \noindent{\it Proof of Claim \ref{lovasz1}.}
 For each $i\in[n]$, let $M_i=\{j\in[n]\setminus\{i\}:v_j\not\in N_i\}$.
 Let $M\s M_i$ be an arbitrary subset.
 Observe that for given $\chi(N_i)$, the event $A_i$ is independent of the event $\bigwedge_{j\in M}A_j$.
 Then

 \begin{equation*}
   \begin{aligned}
     \Pr[A_i\wedge(\bigwedge_{j\in M}A_j)]&=\sum_{y\in[q]^{|N_i|}}\Pr[A_i\wedge(\bigwedge_{j\in M}A_j)|\chi(N_i)=y]
     \Pr[\chi(N_i)=y]\\
     &=\sum_{y\in[q]^{|N_i|}}\Pr[A_i|\chi(N_i)=y]\Pr[\bigwedge_{j\in M}A_j|\chi(N_i)=y]\Pr[\chi(N_i)=y]\\
     &=\fr{1}{q}\sum_{y\in[q]^{|N_i|}}\Pr[\bigwedge_{j\in M}A_j|\chi(N_i)=y]\Pr[\chi(N_i)=y]\\
     &=\Pr[A_i]\Pr[\bigwedge_{j\in M}A_j].\\
   \end{aligned}
 \end{equation*}
 as desired.
\end{proof}
The proof of Theorem \ref{onemoreobservation} is yet another application of the Lov\'asz Local Lemma.

\begin{proof}[\textbf{Proof of Theorem \ref{onemoreobservation}}]
  To prove the theorem we show that $G$ is not $edq^k$-solvable.
  Let $V_{\le k}$ denote the collection of vertices with degree at most $k$ and let $V_{>k}$ denote the remaining set of vertices.
  We have the following claim.

  \begin{claim}\label{lovasz2}
    With the above notation, there exists an assignment of colors to vertices in $V_{\le k}$ with colors in $[edq^k]$, so that all vertices in $V_{\le k}$ guess incorrectly for any possible coloring of the vertices in $V_{>k}$ with colors in $[q]$.
  \end{claim}

   Assuming the correctness of the claim, the theorem follows immediately.
   Indeed, color the vertices of $V_{\le k}$ with the coloring $\chi(V_{\le k})=x_{\le k}$ guaranteed by the claim.
   Since the induced subgraph on $V_{>k}$ is not $q$-solvable, there exists a coloring $\chi(V_{>k})=x_{>k}$ with colors in $[q]$ which makes all the vertices in $V_{>k}$ guess incorrectly given that $\chi(V_{\le k})=x_{\le k}$.
   Therefore, $x_{\le k}$ and $x_{>k}$ form a coloring of the vertices of $G$  for which they all guess incorrectly.

  \vspace{5pt}
  \noindent{\it Proof of Claim \ref{lovasz2}.}
  Without loss of generality, let $V_{\le k}=\{v_1,\ldots,v_m\}$, and for $v_i\in V_{\le k}$ let $N_{i, \leq k}$ and $N_{i, >k}$ be the sets of neighbors of $v_i$ in $V_{\leq k}$ and $V_{>k}$, respectively.	
  Assign independently and uniformly at random to each vertex in $V_{\le k}$ and $V_{> k}$ a hat color from $[edq^k]$ and $[q]$, respectively. 	
  For $1\le i\le m$, let $\ma{A}_i$ be the event that $v_i$ receives a hat color $\chi(v_i)\in [edq^k]$ such that there exists a hat coloring of $N_{i, >k}$ for which $v_i$ guesses its color correctly.
  Let $f_i$ be the guessing strategy of $v_i$, then
	 \begin{equation*}
    \begin{aligned}
   \Pr[\ma{A}_i]&=\sum_{y\in [edq^k]^{|N_{i, \leq k}|}}\Pr[\ma{A}_i|\chi(N_{i, \leq k})=y]\Pr[\chi(N_{i, \leq k})=y]\\
	&=\sum_{y\in [edq^k]^{|N_{i, \leq k}|}}\Pr[\chi(v_i)\in \{f_i\big(\chi(N_{i, \leq k})=y,\chi(N_{i, >k})=x\big):x\in [q]^{|N_{i, >k}|}\}] \frac{1}{(edq^k)^{|N_{i, \leq k}|}}\\
&=\sum_{y\in [edq^k]^{|N_{i,\leq k}|}}
\frac{|\{f_i\big(\chi(N_{i, \leq k})=y,\chi(N_{i, >k})=x\big):x\in [q]^{|N_{i, >k}|}\}|}{edq^k}
\frac{1}{(edq^k)^{|N_{i, \leq k}|}}\\
	&\leq  \sum_{y\in [edq^k]^{|N_{i,\leq k}|}}\frac{q^k}{edq^k} \frac{1}{(edq^k)^{|N_{i, \leq k}|}}\\
	&=  \frac{1}{ed}.	
	    \end{aligned}
  \end{equation*}
  	
Fix $1\le i\le m$ and let $\ma{M}_i=\{j\in[m]:v_j\text{ is of distance at least 3 from }v_i\}$. We will show that $\ma{A}_i$ is mutually independent of the events $\{\ma{A}_j : j\in \ma{M}_i\}.$
Let $\ma{M}\subseteq \ma{M}_i$ be an arbitrary subset, 
then we have the following two observations. $(i)$ For any $y\in[edq^k]^{|N_{i,\le k}|}$, since $N_{i, \leq k}\cap N_{j, \leq k}=\emptyset$ for any $j\in \ma{M} \subseteq \ma{M}_i$, the event $\{\chi(N_{i, \leq k})=y\}$ is independent of the event $\bigwedge_{j\in \ma{M}}\ma{A}_j$. $(ii)$ Given the coloring $\chi(N_{i, \leq k})=y$, the event $\ma{A}_i$ is independent of the event $\bigwedge_{j\in \ma{M}}\ma{A}_j$, since in this case $\ma{A}_i$ happens if and only if $$\chi(v_i)\in\{f_i\big(\chi(N_{i, \leq k})=y,\chi(N_{i, >k})=x\big):x\in [q]^{|N_{i, >k}|}\}.$$ 
Then
\begin{align*}
\Pr[\ma{A}_i|\bigwedge_{j\in\ma{M}}\ma{A}_j)]&
=\sum_{y\in [edq^k]^{|N_{i, \leq k}|}}\Pr[\ma{A}_i|\big(\chi(N_{i, \leq k})=y\big)\wedge \big(\bigwedge_{j\in\ma{M}}\ma{A}_j\big)]
\Pr[\chi(N_{i, \leq k})=y| \bigwedge_{j\in\ma{M}}\ma{A}_j]\\
&=\sum_{y\in [edq^k]^{|N_{i, \leq k}|}}\Pr[\ma{A}_i|\big(\chi(N_{i, \leq k})=y\big)\wedge \big(\bigwedge_{j\in\ma{M}}\ma{A}_j\big)]
\Pr[\chi(N_{i, \leq k}=y]\\
&=\sum_{y\in [edq^k]^{|N_{i, \leq k}|}}\Pr[\ma{A}_i|\chi(N_{i, \leq k})=y]
\Pr[\chi(N_{i, \leq k})=y]\\
&=\Pr[\ma{A}_i],
\end{align*}
where the second and the third equalities follow from $(i)$ and $(ii)$ respectively.

The claim follows by applying Lemma \ref{shearer} with  $p=\fr{1}{ed}$.
\end{proof}

%
%


\section{Linear hat guessing numbers}
\label{sec:linear}

\noindent In this section we consider the case where the players restrict their guessing strategies to  affine functions
of the hat colors assigned to their neighbors, where we always assume that the number of possible hat colors is a prime power.
In what follows we present upper bounds on the linear hat guessing number for various graph families, e.g., cycles, degenerate graphs and graphs with bounded minimum rank.
In particular, in Theorem \ref{Linearforcycles} and Theorem \ref{linearK_n,n}  we show  that  nonlinear guessing strategies can (significantly) outperform linear ones.

Our main tool in this section is the Combinatorial Nullstellensatz \cite{AlonCombi} given next.

\begin{lemma}[Combinatorial Nullstellensatz, see Theorem 1.2, \cite{AlonCombi}]
  Let $\mathbb{F}$ be an arbitrary field, and let $f=f(x_1,\ldots,x_n)$ be a polynomial in
   $\mathbb{F}[x_1,\ldots,x_n]$.
  Suppose the degree $\deg(f)$ of $f$ is $\sum_{i=1}^n t_i$, where each $t_i$ is a nonnegative integer, and suppose the coefficient of $\prod_{i=1}^n x_i^{t_i}$ in $f$ is nonzero.
  Then, if $S_1,\ldots,S_n$ are subsets of $\mathbb{F}$ with $|S_i|>t_i$, there are $s_1\in S_1,s_2\in S_2,\ldots,s_n\in S_n$ so that $$f(s_1,\ldots,s_n)\neq 0.$$
\end{lemma}

As before, let $G$ be a graph on $n$ vertices and let $x\in\mathbb{F}_q^n$ be the coloring assigned to its  vertices.
A linear guessing strategy for vertex  $v_i$ is an affine function of the form  $f_i(x)=\sum_{j\in N_i}a_{i,j}x_j+b_i$, where
$N_i\s[n]$ is  the set of indices of the neighbors of $v_i$, and $a_{i,j},b_i\in\mathbb{F}_q$.
Namely, $v_i$ guesses its color to be $\sum_{j\in N_i}a_{i,j}x_j+b_i$.
The {\it adjacency matrix}  $A^G$ of $G$  is an $n\times n$ binary matrix such that for $1\le i\le n$, $A^G_{i,i}=0$; for $1\le i\neq j\le n$, $A^G_{i,j}=1$ if $v_i,v_j$ are connected, and $A^G_{i,j}=0$ otherwise.
Any linear guessing strategy has a natural matrix representation which is closely related to $A^G$ as follows.
For $1\le i\le n$, let $A^i=(A^i_1,\ldots,A^i_n)\in\mathbb{F}_q^n$ be the vector satisfying $A^i_i=1$, $A^i_j=-a_{i,j}$ for $j\in N_i$, and $A^i_j=0$ for $j\in[n]\setminus(\{i\}\cup N_i)$.
Then, $v_i$ guesses correctly if and only if $\langle A^i,x \rangle=b_i$, where $\langle\cdot\rangle$ is the standard inner product between vectors.
Let $b=(b_1,\ldots,b_n)^t$ and $A$ be the $n\times n$ matrix whose $i$th row is the vector $A^i$ for $1\le i\le n$.
The pair $(A,b)$ is called the {\it matrix representation} of the linear guessing strategy.
The above observation is summarized as the following lemma.

\begin{lemma}\label{matrixrepresentation}
  Any linear guessing strategy for $G$ on the finite field $\mathbb{F}_q$ can be represented by  a pair $(A,b)$, where  $A\in\mathbb{F}_q^{n\times n}$ and  $b=(b_1,\ldots,b_n)^t\in\mathbb{F}_q^n$, such that
  \begin{itemize}
  \item [$(i)$]  $A_{i,i}=1$ for $1\le i\le n$;
  \item [$(ii)$]  $A_{i,j}=0$ if $A^G_{i,j}=0$ and $1\le i\neq j\le n$;
  \item [$(iii)$]  each vertex $v_i$ guesses that its color satisfies the equation $\langle A^i,x \rangle=b_i$, where $A^i$ is the $i$th row of $A$.
\end{itemize}
\end{lemma}

\subsection{Cycles}
\label{subsec:linear_cycles}

\noindent We  begin with a result showing that $C_4$ is linearly 3-solvable.
Although this result was already shown in \cite{Gadouleau09,Szczechla17}, we would like to present its proof as a toy example of a graph whose hat guessing number is attained by a linear guessing strategy.

\begin{example}\label{C_43-solvable}
Label the vertices of $C_4$ successively by $v_1,\ldots,v_4$, and consider the linear guessing strategy represented by the pair $(A,b)$, where
$$A=
 \begin{pmatrix}
 1 & 1 & 0 & 1 \\
 2 & 1 & 1 & 0 \\
 0 & 2 & 1 & 1  \\
 2 & 0 & 2 & 1
 \end{pmatrix}\in\mathbb{F}_3^{4\times 4} \text{, and } b=0\in\mathbb{F}_3^4.$$

\noindent We claim that for any $x\in\mathbb{F}_3^4$, there exist at least one $i\in[4]$ such that $\langle A^i,x\rangle=0$, i.e., $v_i$ guesses  correctly.
Indeed, notice that $A^3=A^1+A^2$ and $A^4=A^1-A^2$, and if $\langle A^1,x\rangle\neq0$ and $\langle A^2,x\rangle\neq0$, then either $\langle A^1,x\rangle+\langle A^2,x\rangle=0$ (and $v_3$ guesses correctly) or $\langle A^1,x\rangle=\langle A^2,x\rangle$ (and $v_4$ guesses correctly). 
\end{example}

We proceed to prove Theorem \ref{Linearforcycles} which claims that $C_3$ and $C_4$ are the only cycles which are linearly 3-solvable.
For the proof we need the following lemma.

\begin{lemma}\label{cyclesminusanedge}
    For $n\ge 3$ if we view $C_n$ as a directed graph with $2n$ directed edges, then deleting any directed edge from it will make the resulting graph, denoted by $C_n^-$ (which has $2n-1$ directed edges), not 3-solvable.
\end{lemma}

\begin{proof}[\textbf{Proof of Theorem \ref{Linearforcycles}}]
 Since it was shown in \cite{Szczechla17} that any cycle is not 4-solvable, to prove Theorem \ref{Linearforcycles}, it suffices to show that for any $n\ge 5$, $C_n$ is not linearly 3-solvable.
  Assume to the contrary that for some $n\ge 5$, $C_n$ is linearly 3-solvable.
  Label the vertices of $C_n$ successively by $v_1,\ldots,v_n$.
  By Lemma \ref{matrixrepresentation} the linear guessing strategy can be represented by a pair $(A,b)$, where $A$ has the following structure

    \begin{table}[h]
    \begin{center}A=
    \begin{tabular}{|c|c|c|c|c|c|c|c|}
      \hline
       & $v_1$ & $v_2$ & $v_3$ & $\cdots$ & $\cdots$ & $v_{n-1}$ & $v_n$ \\\hline
      $v_1$ & 1 & $*$ &  &  &  &  & $*$ \\\hline
      $v_2$ & $*$ & 1 & $*$ &  &  &  &  \\\hline
      $v_3$ &  & $*$ & 1 & $*$ &  &  &  \\\hline
      $\vdots$ &  &  & $\ddots$ & $\ddots$  & $\ddots$ &  &  \\\hline
      $\vdots$ &  &  &  & $\ddots$ & $\ddots$ & $\ddots$ &  \\\hline
      $v_{n-1}$ &  &  &  &  & $*$ & 1 & $*$ \\\hline
      $v_n$ & $*$ &  &  &  &  & $*$ & 1 \\\hline
    \end{tabular}
        \end{center}
    \end{table}
    and the following claims hold:
    \begin{enumerate}
      \item [(1)] the blank cells are filled with zeros;
      \item [(2)] the  symbols $*$ are not  equal to zero;
      \item [(3)] ${\rm{rank}}(A)=n-2$ and every $n-2$ rows of $A$ are linearly independent.
    \end{enumerate}
Indeed, (1) follows from Lemma \ref{matrixrepresentation} $(ii)$ and the adjacency matrix of $C_n$.
(2) holds since if we view $C_n$ as a directed graph with $2n$ directed edges, it is not hard to see that having a $*$  equal to zero is equivalent to deleting  the corresponding directed edge, which  by Lemma  \ref{cyclesminusanedge}  implies that the graph is not even 3-solvable.
It remains to verify (3). To show that the first $n-2$ rows of $A$ are linearly independent, it suffices to consider the $(n-2)\times(n-2)$ submatrix of $A$ formed by rows 1 to $n-2$ and columns 2 to $n-1$. It is a lower triangular matrix with nonzero entries on its diagonal, hence it is invertible. By similar reasoning, one can show that every $n-2$ rows of $A$ are linearly independent.
			
Next by contradiction we show that ${\rm{rank}}(A)\le n-2$. It is obvious that ${\rm{rank}}(A)\neq n$, since otherwise $A$ is invertible and there exists an  $x\in \mathbb{F}_3^n$ with ${\rm{d}}(Ax,b)=n$, i.e., no vertex guesses correctly.
If ${\rm{rank}}(A)=n-1$, then assume  without loss of generality that the first $n-1$ rows of $A$ are linearly independent, and write $A^n=\sum_{i=1}^{n-1}\lambda_i A^i,$ with $\lambda_i\in\mathbb{F}_3$ not all equal to zero since $A^n\neq 0$.
Set $S_i=\mathbb{F}_3\setminus\{b_i\}$ for $i=1,\ldots,n$, and let $j\in [n-1]$ with $\lambda_j\neq 0$.
For $i=1,\ldots,n-1, i\neq j$ let $\gamma_i\in S_i$ be an arbitrary element, and let $x,y \in \mathbb{F}_3^n$ be  two vectors  that satisfy  for $i=1,\ldots,n-1$,
$$\langle A^i,x \rangle=\begin{cases}
\gamma_i & i\neq j \\
w_1 & i=j \\
\end{cases}  \text{ and }\langle A^i,y \rangle=\begin{cases}
\gamma_i & i\neq j \\
w_2 & i=j \\
\end{cases},$$
where $S_j=\{w_1,w_2\}$. It is easy to verify that either   ${\rm{d}}(Ax,b)=n$ or ${\rm{d}}(Ay,b)=n$ (possibly both), i.e., there is a hat assignment for which  no vertex guesses correctly.
Hence the rank of $A$ is $n-2$ and any set of $n-2$ rows are linearly independent. In particular, the rows $A^{n-1},A^n$ can be written as
$$A^{n-1}=\sum_{i=1}^{n-2}\lambda_iA^i,\qquad A^n=\sum_{i=1}^{n-2}\mu_i A^i,$$
with $\lambda_i,\mu_i\neq 0$ for any $i$,  since every set of $n-2$ rows of $A$ are linearly independent.

Let $A'$ be the submatrix of $A$ formed by its first $n-2$ rows.
For $1\le i\le n-2$, define the linear form $z_i = \langle A^i, x\rangle $, and by linearity
$$\langle A^{n-1},x\rangle=\sum_{i=1}^{n-2} \lambda_iz_i,\qquad \langle A^{n},x\rangle=\sum_{i=1}^{n-2}\mu_i z_i.$$
Observe that since $A'$ is of full rank then the mapping $x\mapsto A'x, x\in \mathbb{F}_3^n$ is onto $\mathbb{F}_3^{n-2}$ and in particular, every $(z_1,\ldots,z_{n-2})\in S_1\times\cdots\times S_{n-2}$ is
contained in its image.		
By Lemma \ref{matrixrepresentation} $(iii)$ and \eqref{guessingstrategy} the polynomial
$$\prod_{i=1}^n(\langle A^i,x\rangle-b_i)=
\prod_{i=1}^{n-2}(z_i-b_i)(\sum_{i=1}^{n-2}\lambda_iz_i-b_{n-1})(\sum_{i=1}^{n-2}\mu_iz_i-b_n)$$
vanishes on $\mathbb{F}_3^n$.
In particular,
   $$h(z_1,\ldots,z_{n-2}):=(\sum_{i=1}^{n-2}\lambda_iz_i-b_{n-1})(\sum_{i=1}^{n-2}\mu_iz_i-b_n)$$
vanishes on $S_1\times\cdots\times S_{n-2}$.
Since $\deg(h)=2$ and $|S_i|\ge 2$, by the Combinatorial Nullstellensatz it is not hard to see that the coefficients of the cross terms $z_iz_j,1\le i\neq j\le n-2$ must be all zero. In particular, since  $n-2\geq 3$  then the  coefficients of $z_1z_2,z_2z_3$, and $z_1z_3$ satisfy
\begin{equation*}
    \left\{
        \begin{aligned}
           \lambda_1\mu_2+\lambda_2\mu_1=0,\\
        \lambda_2\mu_3+\lambda_3\mu_2=0,\\
            \lambda_1\mu_3+\lambda_3\mu_1=0.\\
        \end{aligned}
    \right.
\end{equation*}
Recall that  none of the $\lambda_i$'s and $\mu_i$'s is equal to zero, hence
\begin{equation*}
    \left\{
        \begin{aligned}
  \fr{\lambda_1}{\lambda_2}=-\fr{\mu_1}{\mu_2},\\
   \fr{\lambda_2}{\lambda_3}=-\fr{\mu_2}{\mu_3},\\
    \fr{\lambda_3}{\lambda_1}=-\fr{\mu_3}{\mu_1}.\\
        \end{aligned}
    \right.
\end{equation*}
Multiplying all the left hand sides and all the right hand sides of the above equations we get $1=-1$, a contradiction.
\end{proof}


To conclude the proof of the theorem, it remains to prove Lemma \ref{cyclesminusanedge}.

\begin{proof}[\textbf{Proof of Lemma \ref{cyclesminusanedge}}]
Label the $n$ vertices of $C_n^-$ successively by $v_1,\ldots,v_n$.
As before, for $1\le i\le n$, let $x_i$ and $f_i$ be the hat color and the guessing strategy of $v_i$, respectively.
Without loss of generality, assume that $v_1\rightarrow v_n$ is the deleted directed edge, hence $v_1$ cannot see the hat color assigned to $v_n$, however $v_n$ can see the hat color assigned to $v_1$.
Therefore, one can write $f_1=f_1(x_2)$, $f_i=f_i(x_{i-1},x_{i+1})$ for $2\le i\le n-1$ and $f_n=f_n(x_{n-1},x_1)$.
Our goal is to find an assignment of hat colors $x=(x_1,\ldots,x_n)\in [3]^n$ for which all the vertices guess incorrectly, i.e., $f_i\neq x_i$ for each $1\le i\le n$.
We will pick the colors $x_i,i=1,\ldots,n$ successively starting from $x_1$.
Set $x_1$ to be a color in $[3]$ such that $$|\{c\in [3]:f_1(c)= x_1\}|\le 1.$$
It is easy to verify that there are in fact at least two possible choices for $x_1$.
Assuming that colors $x_1,\ldots,x_{i-1}$ were already assigned, assign the color $x_i$ as follows.
By the $(i-1)$th hat assignment,
$$|\{c\in [3]:f_{i-1}(x_{i-2},c)= x_{i-1}\}|\le 1 \text{ hence }|\{c\in [3]:f_{i-1}(x_{i-2},c)\neq x_{i-1}\}|\geq 2.$$
Let $A\subseteq [3]$ be the set of colors $x_i$ for which
 $$|\{c\in [3]:f_i(x_{i-1},c)= x_i\}|\le 1.$$
Given $x_{i-1}$ one can check that $|A|\geq 2$. Set $x_i$ to be any color in the intersection $$A\cap \{c\in [3]:f_{i-1}(x_{i-2},c)\neq x_{i-1}\},$$ which is nonempty since  these  are two subsets of $[3]$, each of size  at least $2$.
Lastly, we set the color $x_n$ as follows. By the $(n-1)$th hat color assignment,
$$|\{c\in [3]:f_{n-1}(x_{n-2},c)= x_{n-1}\}|\leq 1 \text{ hence } |\{c\in [3]:f_{n-1}(x_{n-2},c)\neq  x_{n-1}\}|\geq  2.$$
Clearly we have
$$|\{c\in [3]:f_{n}(x_{n-1},x_1)\neq c\}|= 2.$$
Set $x_n$ to be any color in the nontrivial intersection
$$\{c\in [3]:f_{n-1}(x_{n-2},c)\neq  x_{n-1}\}\cap \{c\in [3]:f_{n}(x_{n-1},x_1)\neq c\}.$$
It is easy to verify that by the above color assignments all the vertices make incorrect guesses.
\end{proof}

\subsection{Complete bipartite graphs}
\label{subsec:linear_bipartite}

\noindent In this subsection we prove Theorem \ref{linearK_n,n}. We need the following lemma given in \cite{Alon89}.

\begin{lemma}[see Proposition 1, \cite{Alon89}]\label{alon89}
  Let $A$ be a nonsingular $n\times n$ matrix over $\mathbb{F}_q$, where $q\ge 4$ is a proper prime power with characteristic $p$.
  Let $S_1,\ldots,S_n\s\mathbb{F}_q$ be arbitrary subsets, each of cardinality $p$, and let $s=(s_1,\ldots,s_n)$ be an arbitrary vector of $\mathbb{F}_q^n$.
  Then there exists a vector $x=(x_1,\ldots,x_n)\in S_1\times\cdots\times S_n$ such that
  ${\rm{d}}(Ax,s)=n$.
\end{lemma}

To prove Theorem \ref{linearK_n,n}, we use the following simple consequence of Lemma \ref{alon89}.

\begin{lemma}\label{modify}
  For positive integers $m\le n$, let $A'$ be an $m\times n$ matrix over $\mathbb{F}_q$ with full row rank, where $q\ge 4$ is a proper prime power with characteristic $p$. Let $S_1,\ldots,S_n\s\mathbb{F}_q$ be arbitrary subsets, each of cardinality $p$, and let $s=(s_1,\ldots,s_m)$ be an arbitrary vector of $\mathbb{F}_q^m$.
  Then there exists a vector $x=(x_1,\ldots,x_n)\in S_1\times\cdots\times S_n$ such that
  ${\rm{d}}(A'x,s)=m$.
\end{lemma}

\begin{proof}
  Complete $A'$ to a nonsingular $n\times n$ matrix $A$ and choose  $s_{n-m+1},\ldots,s_n\in \mathbb{F}_q$ arbitrarily.
  Then the result follows easily from Lemma \ref{alon89}.
\end{proof}

\begin{proof}[\textbf{Proof of Theorem \ref{linearK_n,n}}] 
Assume that $q$ is a power of the prime number $p$. 
Let $V_L$ and $V_R$ be the  left and right vertex parts of the graph, respectively.
Let $\chi(V_L)=x=(x_1,x_2,\ldots,x_n)$ and $\chi(V_R)=y=(y_1,y_2,\ldots,y_n)$, where $x,y\in\mathbb{F}_q^n$.
By Lemma \ref{matrixrepresentation}, the linear guessing strategy of $K_{n,n}$ can be represented by the following $2\times 2$ block matrix

$$
 A=\begin{pmatrix}
 I_n & B   \\
 C & I_n
 \end{pmatrix} \text{, and the vector } (b,c),$$

\noindent where $I_n$ is the $n\times n$ identity matrix, $B,C$ are $n\times n$ matrices, and  $b,c\in\mathbb{F}_q^n$.

$K_{n,n}$ is not linearly $q$-solvable if and only if there exist color assignments $x, y\in\mathbb{F}_q^n$ such that
$${\rm{d}}\big(A\cdot \begin{pmatrix}
 x  \\
 y
 \end{pmatrix},(b,c)^t\big)=n,$$
or equivalently for $i=1,\ldots,n$

\begin{equation}\label{3}
  \begin{aligned}
    x_i+\langle B^i,y\rangle \neq b_i, \qquad  y_i+\langle C^i,x\rangle \neq c_i,
  \end{aligned}
\end{equation}
where $B^i$ and $C^i$ are the $i$th rows of $B$ and $C$, respectively.
Define $2n$ linear forms $z_i$ in the variables $x_i,y_i$ as follows
\begin{equation}
\label{2nforms}
z_i=\begin{cases}
x_i+\langle B^i,y\rangle  & i=1,\ldots,n\\
y_{i-n}+\langle C^{i-n},x\rangle & i=n+1,\ldots,2n,
\end{cases}
\end{equation}
and note that $\{z_1,\ldots,z_n\}$ and $\{z_{n+1},\ldots,z_{2n}\}$ are two sets of linearly independent forms.
Complete the first $n$ forms to a basis $z_1,z_2,\ldots,z_n,z_{n+1},\ldots,z_{n+k}$ of these $2n$ forms, where we assume without loss of generality that the first $n+k$ forms among $z_1,\ldots,z_{2n}$ form such a basis.
Notice that possibly $k=0$.
The remaining $n-k$ forms $z_i$ for $i\in \{n+k+1,\ldots,2n\}$ are linearly independent, and each of them is a linear combination of the first $n+k$ forms.
Let $A'$ be the $(n-k)\times(n+k)$ coefficient matrix of these $n-k$ forms, i.e., $A'\cdot(z_1,\ldots,z_{n+k})^t=(z_{n+k+1},\ldots,z_{2n})$.
Applying Lemma \ref{modify} with $A'$, $s_i=c_{k+i}$ for $i=1,\ldots,n-k$, and $S_j$ being an arbitrary $p$-subset of $\mathbb{F}_q\setminus\{b_j\}$ for $j=1,\ldots,n$ and of $\mathbb{F}_q\setminus\{c_{j-n}\}$ for $j=n+1,\ldots,n+k$ (such $S_j$'s do exist since $q$ is a proper prime power of $p$), it follows that there exist $z_1,\ldots,z_{n+k}\in\mathbb{F}_q$ such that $z_j\neq b_j$ for $1\le j\le n$, $z_j\neq c_{j-n}$ for $n+1\le j\le n+k$ and
\begin{equation}\label{5}
  \begin{aligned}
    {\rm{d}}\big(A'\cdot(z_1,\ldots,z_{n+k})^t,(c_{k+1},\ldots,c_n)^t\big)=n-k.
  \end{aligned}
\end{equation}
The result follows directly from (\ref{3}), (\ref{2nforms}) and (\ref{5}).
\end{proof}

\subsection{Degenerate graphs}
\label{subsec:degenerate}

\begin{proof}[\textbf{Proof of Theorem \ref{degeneratenonlinear}}]
  Let $G$ be a graph on $n$ vertices and let $v_1,\ldots,v_n$ be an ordering of its vertices such that for each $2\le i\le n$, $v_i$ is connected to at most $d$ vertices among $v_1,\ldots,v_{i-1}$.
  We prove the theorem by contradiction.
  It suffices to show that $G$ is not linearly $q$-solvable for any prime power $q\ge d+2$.
  For $1\le i\le n$, let $x_i\in \mathbb{F}_q$ and the affine function $f_i$ be the hat color and the guessing strategy of $v_i$, respectively.
  By (\ref{guessingstrategy}),
  $$F(x):=\prod_{i=1}^n(x_i-f_i)$$ vanishes on $\mathbb{F}_q^n$. 
  For $1\le i,j\le n$, let $p_i=x_i-f_i$ and write $x_j\in p_i$ if the monomial $x_j$ appears in  $p_i$ with a  nonzero coefficient.
  Next, we define successively $n$  subsets $P_n,\ldots,P_1\s\{p_1,\ldots,p_n\}$  as follows.
  Let $P_n=\{p_j:x_n\in p_j\}$, and for  $i=n-1,\ldots,1$,  $$P_i=\{p_j\not\in\cup_{k=i+1}^n P_k:x_i\in p_j\}.$$
  Notice that  some of the $P_i$'s might be the empty set, and  $\cup_{i=1}^n P_i=\{p_1,\ldots,p_n\}$.
  Furthermore, by the $d$-degeneracy of $G$, $|P_i|\le d+1$ for any $i$.
  Rewrite $F$ as
  $$F(x)=\prod_{i=1}^n\prod_{p\in P_i} p(x),$$
  and  consider its monomial $m(x):=x_n^{|P_n|}x_{n-1}^{|P_{n-1}|}\cdots x_1^{|P_1|}$.
  Obviously, $\deg m(x)=\sum_{i=1}^n |P_i|=n=\deg F$.
  Moreover, the coefficient of $m(x)$ must be nonzero, since the coefficient of $x_i^{|P_i|}$ in $\prod_{p\in P_i} p(x)$ is nonzero, and $m(x)$ is formed only by successively multiplying the monomials $x_i^{|P_i|}$ in $\prod_{p\in P_i} p(x)$ for $i=n,\ldots,1$.
  Therefore, by the Combinatorial Nullstellensatz  $F(x)$ does not vanish on $\mathbb{F}_q^n$ since $q\geq d+2$.
\end{proof}

\subsection{Linear solvability and min-rank}
\label{subsec:linear_min_rank}

\begin{proof}[\textbf{Proof of Theorem \ref{minrk1}}]
Similar to the proof of Theorem \ref{degeneratenonlinear}, it suffices to show that $G$ is not linearly $q$-solvable for any prime power $q\ge n-{\rm{mr}}(G)+2$.
Assume the opposite.
By Lemma \ref{matrixrepresentation} let the matrix $A \in \mathbb{F}_q^{n\times n}$ and the vector $b\in \mathbb{F}_q^n$ represent the linear guessing strategy of $G$, where vertex $v_i$ guesses $\langle A^i,x\rangle=b _i$, and $A^i$ is the $i$th row of $A$.
It is clear  from  Lemma \ref{matrixrepresentation}  $(i),(ii)$ that $A$ fits $G$.
Suppose ${\rm{rank}}(A)=r$, then by definition ${\rm{mr}}(A)\le r$.

Without loss of generality, assume that the first $r$ rows of $A$ are linearly independent, then  for $j\in\{r+1, \ldots, n\}$ write $A^j=\sum_{i=1}^r\lambda_{ij}A^i$, and by linearity
$\langle A^j,x \rangle = \sum_{i=1}^r\lambda_{ij} \langle A^i,x \rangle$.
Let $A'$ be the submatrix of $A$ formed by its first $r$ rows.
Since $A'$ is of full row rank,  $A'x$ is onto $\mathbb{F}_q^r$ as $x$ ranges over $\mathbb{F}_q^n$.

For $i=1,\ldots,n$, let $z_i=\langle A^i,x \rangle$ be a nonzero linear form, since $A^i\neq 0$.
By Lemma \ref{matrixrepresentation} $(iii)$ and \eqref{guessingstrategy}
$$\prod_{i=1}^n(\langle A^i,x \rangle-b_i)=\prod_{i=1}^r(z_i-b_i)\prod_{j=r+1}^n(\sum_{i=1}^r\lambda_{ij}z_i-b_j)$$
vanishes on $\mathbb{F}_q^n$. 	
Denote $$f(z_1,\ldots,z_r):=\prod_{j=r+1}^n(\sum_{i=1}^r\lambda_{ij}z_i-b_j),$$
and let $S_i=\mathbb{F}_q\setminus\{b_i\}$ for $1\le i\le r$, then $f$ vanishes on every
$(z_1,\ldots,z_r)\in\ S_1\times\cdots\times S_r$.
For $r+1\le j\le n$, since $A^j\neq 0$ then $\lambda_{1j},\ldots,\lambda_{rj}$ are not all zero.
Set $j^*$ to be the smallest index $i$ such that  $\lambda_{ij}\neq 0$, and notice that the coefficient of the monomial
$\prod_{j=r+1}^n z_{j^*}$ in $f(z_1,\ldots,z_r)$ is nonzero.
Finally, since $\deg f=n-r$, $|S_1|=\cdots=|S_r|=q-1$, we have that $n-r\le n-{\rm{mr}}(A)<q-1=|S_i|$, and by the Combinatorial Nullstellensatz $f$ cannot vanish on $S_1\times\cdots\times S_r$, a contradiction.

\end{proof}

\noindent\textbf{Note added in proof:} After posting the paper in the arXiv we learned from Oleg Pikhurko that Ostap Chervak \cite{Chervak} proved
independently some of our results including Theorem 1.2 for $r=2$ and Theorem 1.3 with similar parameters.
We thank Oleg for pointing this out.

\section*{Acknowledgements}
\noindent The research of Noga Alon was supported by NSF grant DMS-1855464, ISF grant 281/17, BSF grant 2018267 and the Simons Foundation.
The research of Chong Shangguan and Itzhak Tamo was supported by ISF grant No. 1030/15 and NSF-BSF grant No. 2015814.

\bibliographystyle{plain}

\bibliography{solvability}

\appendix
\section{Proof of Lemma \ref{saturated}}
\label{sec:proof_lem_saturated}

\noindent The proof is a standard application of the alteration method, (see e.g. Chapter 3 of \cite{probmethod}).
Construct an $n\times 2l$ $q$-ary matrix $M$ by picking each entry independently and  uniformly at random   from $[q]$.
We say that $S$, a subset of $t$ columns of $M$, is a bad $t$-tuple if there exists no row whose restriction to $S$ is onto $[q]$, i.e., for any $r\in [n]$
$$r(S):=\{M_{r,s}:s\in S\}\neq [q],$$
where $M_{r,s}$ is the entry of $M$ in row $r$ and column $s$.
For part $(i)$, let $S$ be a subset of size $q$, then for any $r$

$$\Pr[r(S)\neq[q]]=1-\fr{q!}{q^q}, \text{ hence } \Pr[S~is~bad]=(1-\fr{q!}{q^q})^n.$$

\noindent Therefore the expected number of bad $q$-tuples is bounded from above by $(2l)^q(1-\fr{q!}{q^q})^n$,
and there exists a matrix  with at most this many bad $q$-tuples.
Delete one column from each bad $q$-tuple of the matrix.
Clearly, if $(2l)^q(1-\fr{q!}{q^q})^n\le l$ then the resulting matrix is $q$-saturated with at least $l$ columns, as desired.

The proof of part $(ii)$ is similar.
Let $S$ be a fixed subset of $t\geq q$ columns of $M$, then for any $r\in[n]$

$$\Pr[r(S)\neq[q]]=\Pr[\exists~a\in[q],~s.t.~a\not\in r(S)]\le q(1-\fr{1}{q})^t.$$

\noindent Note that the bound for the probability of a bad $t$-tuple is different from the one given in the proof of part $(i).$
Thus, $$\Pr[S~is~bad]\le(q(1-\fr{1}{q})^t)^n=q^n(1-\fr{1}{q})^{tn},$$
\noindent and the expected number of bad $q$-tuples is bounded from above by $(2l)^tq^n(1-\fr{1}{q})^{tn}.$
The rest of the proof is identical to the proof of part $(i)$, and therefore is omitted.


\section{Proof of Theorem \ref{trees}}
\label{sec:proof_thm_trees}

\begin{proof}[\textbf{Proof of Theorem \ref{trees}}]
Assume $G$ has $n$ vertices $v_1,\ldots,v_n$, with $\deg(v_1)=1$ and $v_2$ is the only neighbor of $v_1$.
Let the set of $q$ colors be $[q]$.
For $1\le i\le n$, let $f_i:\mathbb{R}^{|N_i|}\rightarrow\mathbb{R}$ be the function which represents the guessing strategy of $v_i$, where $N_i$ is the set of neighbors of $v_i$ in $G$.
By (\ref{guessingstrategy})
$$F(x)=\prod_{i=1}^n\big(x_i-f_i(x)\big)$$
vanishes on $[q]^n$.
By our assumption on $v_1$,  $F(x)$ can be written as

$$F(x)=\big(x_1-f_1(x_2)\big)\big(x_2-f_2(x_1,x_{i_1},\ldots,x_{i_d})\big)H(x_2,\ldots,x_n),$$

\noindent where $ I:=\{i_1,\ldots,i_d\}$ is the set of indices for which $v_{i_1},\ldots,v_{i_d}$ are the neighbors of $v_2$ in $G\setminus\{v_1\}$ and $H(x_2,\ldots,x_n)=\prod_{i=3}^n\big(x_i-f_i(x)\big)$ does not depend on $x_1$.
The result will follow by constructing a guessing strategy $f'_2(x_{i_1},\ldots,x_{i_d})$ for $v_2$ in the graph $G\setminus\{v_1\}$ such that

\begin{equation}\label{fomulatree2}
	F'(x_2,\ldots,x_n)=\big(x_2-f'_2(x_{i_1},\ldots,x_{i_d})\big)H(x_2,\ldots,x_n)=0
\end{equation}
	
\noindent for any $(x_2,\ldots,x_n)\in [q]^{n-1}$.
	
For a function $f:\mathbb{R}^n\rightarrow \mathbb{R}$, let $Z(f)=\{x\in \mathbb{R}^n:f(x)=0\}$ be its zero set, and denote $A=\{(a,0,\ldots,0):a\in[q]\}$.
If $[q]^{n-1}\s Z(H)$ then there is nothing to prove, since for any function $f'_2$ \eqref{fomulatree2} holds.
Otherwise, for any $(a_2,\ldots,a_n)\in [q]^{n-1}\backslash Z(H)$ the set $A+(0,a_2,\ldots,a_n)=\{(a,a_2,\ldots,a_n):a\in[q]\}$ satisfies
$$\big(A+(0,a_2,\ldots,a_n)\big)\s\cup_{i=1}^2Z(x_i-f_i).$$
Indeed, since $F(x)$ vanishes on $[q]^n$ then $\big(A+(0,a_2,\ldots,a_n)\big)\s\cup_{i=1}^2Z(x_i-f_i)\cup Z(H)$.
We claim that $\big(A+(0,a_2,\ldots,a_n)\big)\cap Z(H)=\emptyset$.
Otherwise, if $(a_1,a_2,\ldots,a_n)\in Z(H)$ for some $a_1\in [q]$ then $H(a_2,\ldots,a_n)=0$, which is a contradiction.
Observe also that for any set $A+(0,a_2,\ldots,a_n)$ the following holds

$$\big(A+(0,a_2,\ldots,a_n)\big)\cap Z(x_1-f_1)=\big(f_1(a_2),a_2,\ldots,a_n\big),$$

\noindent hence if  $(a_2,\ldots,a_n)\not\in Z(H)$ then

\begin{equation}\label{fomulatree3}
	\begin{aligned}
	|\big(A+(0,a_2,\ldots,a_n)\big)\cap Z(x_2-f_2)|\ge q-1.
	\end{aligned}
\end{equation}
	
Consider the set $\ma{C}=[q]^{n-1}\backslash Z(H)$.
A guessing strategy $f'_2$ which satisfies $\ma{C}\s Z(x_2-f'_2)$ is sufficient for (\ref{fomulatree2}) to hold, since then   $Z(x_2-f'_2)\cup Z(H)=[q]^{n-1}$.	
The value of  $f'_2(x_{i_1},\ldots,x_{i_d})$ at the point $(a_{i_1},\ldots,a_{i_d})$ is defined as follows.

\begin{itemize}
		\item [$(i)$] If there exists $(a_2,\ldots,a_n)\in\ma{C}$ such that $(a_2,\ldots,a_n)|_I=(a_{i_1},\ldots,a_{i_d})$, set $f'_2(a_{i_1},\ldots,a_{i_d})=a_2$;
		\item [$(ii)$] Otherwise, set $f'_2(a_{i_1},\ldots,a_{i_d})\in[q]$ arbitrarily.
\end{itemize}
		
It remains to verify that $(1)$ $f'_2$ is a well-defined function; $(2)$ it satisfies $Z(x_2-f'_2)\cup Z(H)=[q]^{n-1}$.
For $(1)$ assume by contradiction that it is not well-defined, and $\ma{C}$ contains two vectors $(a_2,a_3\ldots,a_n),(a'_2,a'_3,\ldots,a'_n)\in[q]^{n-1}$ satisfying $a_{i} = a'_{i}$ for $i \in I$ and $a_2\neq a'_2$.
Let
$$\Lambda=\{a\in[q]: (a,a_2,\ldots,a_n)\in \big(A+(0,a_2,\ldots,a_n)\big)\cap Z(x_2-f_2)\}$$
and
$$\Lambda'=\{a\in[q]: (a,a_2',\ldots,a_n')\in \big(A+(0,a_2',\ldots,a_n')\big)\cap Z(x_2-f_2)\}.$$
By (\ref{fomulatree3})  $\Lambda,\Lambda'\subseteq [q]$  are of  size at least $q-1$, therefore $\Lambda\cap\Lambda'\neq\emptyset$.
Let $a_1\in\Lambda\cap\Lambda'\s [q]$ be an arbitrary element.
By definition, $Z(x_2-f_2)$ contains both $(a_1,a_2,a_3,\ldots,a_n)$ and $(a_1,a'_2,a'_3,\ldots,a'_n)$, which implies
$a_2=f_2(a_1,a_{i_1},\ldots,a_{i_d})=f_2(a_1,a'_{i_1},\ldots,a'_{i_d})=a'_2$, which is a contradiction and hence $f'_2$ is well-defined.
For $(2)$ if $(a_2,\ldots,a_n)\in\ma{C}$, $(i)$ implies that $a_2=f'_2(a_{i_1},\ldots,a_{i_d})$, or equivalently, $(a_2,\ldots,a_n)\in Z(x_2-f'_2)$, as desired.
We conclude that \eqref{fomulatree2} holds and the polynomials  $f'_2, H$ form a proper guessing strategy for the graph $G\setminus\{v_1\}$.
\end{proof}

It would be interesting to extend the proof scheme to more general cases, say, $2$-degenerate graphs.
 \end{document}